\documentclass[11pt,francais]{article}
\usepackage{amsmath}\usepackage{epsf,amsfonts,amsthm}\usepackage{amscd,mathptmx,amssymb}
\usepackage{xcolor,epic,eepic}\usepackage{epsfig}
\usepackage{fontenc,indentfirst, delarray,amsfonts,amsmath,amssymb}
\usepackage{rotating}
\usepackage{mathdots}
\usepackage[T1]{fontenc}
\usepackage[matrix,arrow,curve]{xy}
\usepackage{amsmath}
\usepackage{amssymb}
\usepackage{amsthm}
\usepackage{amscd}
\usepackage{amsfonts}
\usepackage{graphicx}%
\usepackage{fancyhdr}
\usepackage{dsfont,texdraw}
\usepackage{amsmath}\usepackage{epsf,amsfonts,amsthm}\usepackage{amscd,mathptmx,amssymb}
\usepackage{xcolor,epic,eepic}\usepackage{epsfig}
\usepackage{fontenc,indentfirst, delarray,amsfonts,amsmath,amssymb}
\usepackage{rotating}
\usepackage[T1]{fontenc}

\usepackage{enumerate}                   

\usepackage{amssymb,mathrsfs}            
\usepackage[matrix,arrow,curve]{xy}      
\usepackage{bm}                          
\usepackage{color}                       

\newtheorem{rem}{Remark}
\newtheorem{theo}{Theorem}
\newtheorem{prop}{Proposition}

\newtheorem{cor}{Corollary}
\newtheorem{ex}{Example}
\newtheorem{defi}{Definition}

\newcommand{\bea}{\begin{eqnarray}}
\newcommand{\eea}{\end{eqnarray}}
\newcommand{\nn}{\nonumber}
\def\ga{\alpha}

\def\gd{\delta}

\def\ep{\epsilon}

\def\go{\omega}
\newcommand{\SC}{\mathscr}

\newcommand{\be}{\begin{equation}}
\newcommand{\ee}{\end{equation}}
\newcommand{\w}{\wedge}
\newcommand{\p}{\partial}

\newcommand{\Ci}{C^{\infty}}

\newcommand{\N}{\mathbb{N}}
\newcommand{\Z}{\mathbb{Z}}
\newcommand{\R}{\mathbb{R}}
\newcommand{\K}{\mathbb{K}}

\newcommand{\lp}{\left(}
\newcommand{\rp}{\right)}
\newcommand{\op}[1]{\!\!\mathop{\rm ~#1}\nolimits}

\mathchardef\za="710B  
\mathchardef\zb="710C  
\mathchardef\zg="710D  
\mathchardef\zd="710E  
\mathchardef\zve="710F 
\mathchardef\zz="7110  
\mathchardef\zh="7111  
\mathchardef\zy="7112 

\mathchardef\zi="7113  
\mathchardef\zk="7114  
\mathchardef\zl="7115  
\mathchardef\zm="7116  
\mathchardef\zn="7117  
\mathchardef\zx="7118  
\mathchardef\zp="7119  
\mathchardef\zr="711A  
\mathchardef\zs="711B  
\mathchardef\zt="711C  
\mathchardef\zu="711D  
\mathchardef\zf="711E 
\mathchardef\zq="711F  
\mathchardef\zc="7120  
\mathchardef\zw="7121  
\mathchardef\ze="7122  
\mathchardef\zvy="7123  
\mathchardef\zvw="7124  
\mathchardef\zvr="7125 
\mathchardef\zvs="7126 
\mathchardef\zvf="7127  
\mathchardef\zG="7000  
\mathchardef\zD="7001  
\mathchardef\zY="7002  
\mathchardef\zL="7003  
\mathchardef\zX="7004  
\mathchardef\zP="7005  
\mathchardef\zS="7006  
\mathchardef\zU="7007  
\mathchardef\zF="7008  
\mathchardef\zW="700A  

\newcommand{\h}{\op{Hom}}
\newcommand{\0}{\otimes}
\newcommand{\ac}{\,\text{!`}}

\usepackage{tikz}
\usetikzlibrary{arrows,chains,matrix,positioning,scopes,snakes}

\makeatletter
\tikzset{join/.code=\tikzset{after node path={%
\ifx\tikzchainprevious\pgfutil@empty\else(\tikzchainprevious)%
edge[every join]#1(\tikzchaincurrent)\fi}}}
\makeatother

\tikzset{>=stealth',every on chain/.append style={join},
         every join/.style={->}}
\tikzset{
    >=stealth',
    punkt/.style={
           rectangle,
           rounded corners,
           draw=black, very thick,
           text width=6.5em,
           minimum height=2em,
           text centered},
    pil/.style={
           ->,
           thick,
           shorten <=2pt,
           shorten >=2pt,}
}

\begin{document}
\title{{\bf On the infinity category of homotopy Leibniz algebras}}
\date{}
\author{David Khudaverdyan, Norbert Poncin, Jian Qiu\thanks{The research of D. Khudaverdyan and N. Poncin was supported by Grant
GeoAlgPhys 2011-2013 awarded by the University of Luxembourg. J. Qiu is grateful to the Luxembourgian National Research Fund for support via AFR grant PDR 2011-2. The authors thank B. Shoikhet for an informative discussion about composition of homotopies in the DGLA case, as well as J. Nogami and G. Bonavolont\`a for explanations on infinity categories.}
}
\maketitle

\begin{abstract} We discuss various concepts of $\infty$-homotopies, as well as the relations between them (focussing on the Leibniz type). In particular $\infty$-$n$-homotopies appear as the $n$-simplices of the nerve of a complete Lie ${\infty}$-algebra. In the nilpotent case, this nerve is known to be a Kan complex \cite{Get09}. We argue that there is a quasi-category of $\infty$-algebras and show that for truncated $\infty$-algebras, i.e. categorified algebras, this $\infty$-categorical structure projects to a strict 2-categorical one. The paper contains a shortcut to $(\infty,1)$-categories, as well as a review of Getzler's proof of the Kan property. We make the latter concrete by applying it to the 2-term $\infty$-algebra case, thus recovering the concept of homotopy of \cite{BC04}, as well as the corresponding composition rule \cite{SS07}. We also answer a question of \cite{BS07} about composition of $\infty$-homotopies of $\infty$-algebras.
\end{abstract}

\maketitle \vspace{2mm}\noindent {\bf{Keywords}} : Infinity algebra, categorified algebra, higher category, quasi-category, Kan complex, Maurer-Cartan equation, composition of homotopies, Leibniz algebra

\tableofcontents

\section{Introduction}

\subsection{General background}

Homotopy, sh, or infinity algebras \cite{Sta63} are homotopy invariant extensions of differential graded algebras. They are of importance, e.g. in {\small BRST} of closed string field theory, in Deformation Quantization of Poisson manifolds ... Another technique to increase the flexibility of algebraic structures is categorification \cite{CF94}, \cite{Cra95} -- a sharpened viewpoint that leads to astonishing results in {\small TFT}, bosonic string theory ... Both methods, homotopification and categorificiation are tightly related: the 2-categories of 2-term Lie (resp., Leibniz) infinity algebras and of Lie (resp., Leibniz) 2-algebras turned out to be equivalent \cite{BC04}, \cite{SL10} (for a comparison of 3-term Lie infinity algebras and Lie 3-algebras, as well as for the categorical definition of the latter, see \cite{KMP11}). However, homotopies of $\infty$-morphisms and their compositions are far from being fully understood. In \cite{BC04}, $\infty$-homotopies are obtained from categorical homotopies, which are God-given. In \cite{SS07}, (higher) $\infty$-homotopies are (higher) derivation homotopies, a variant of infinitesimal concordances, which seems to be the wrong concept \cite{DP12}. In \cite{BS07}, the author states that $\infty$-homotopies of sh Lie algebra morphisms can be composed, but no proof is given and the result is actually not true in whole generality. The objective of this work is to clarify the concept of (higher) $\infty$-homotopies, as well as the problem of their compositions.

\subsection{Structure and main results}

In Section 2, we provide explicit formulae for Leibniz infinity algebras and their morphisms. Indeed, although a category of homotopy algebras is simplest described as a category of quasi-free {\small DG} coalgebras, its original nature is its manifestation in terms of brackets and component maps.\medskip

We report, in Section 3, on the notions of homotopy that are relevant for our purposes: concordances, i.e. homotopies for morphisms between quasi-free {\small DG} (co)algebras, gauge and Quillen homotopies for Maurer-Cartan ({\small MC} for short) elements of pronilpotent Lie infinity algebras, and $\infty$-homotopies, i.e. gauge or Quillen homotopies for $\infty$-morphisms viewed as {\small MC} elements of a complete convolution Lie infinity algebra.\medskip

Section 4 starts with the observation that vertical composition of $\infty$-homotopies of {\small DG} algebras is well-defined. However, this composition is not associative and cannot be extended to the $\infty$-algebra case -- which suggests that $\infty$-algebras actually form an $\infty$-category. To allow independent reading of the present paper, we provide a short introduction to $\infty$-categories, see Subsection \ref{InftyCatIntro}. In Subsection \ref{InftyCatInftyAlg}, the concept of $\infty$-$n$-homotopy is made precise and the class of $\infty$-algebras is viewed as an $\infty$-category. Since we apply the proof of the Kan property of the nerve of a nilpotent Lie infinity algebra to the 2-term Leibniz infinity case, a good understanding of this proof is indispensable: we detail the latter in Subsection \ref{Kan property}.\medskip

To be complete, we give an explicit description of the category of 2-term Leibniz infinity algebras at the beginning of Section 5. We show that composition of $\infty$-homotopies in the nerve-$\infty$-groupoid, which is defined and associative only up to higher $\infty$-homotopy, projects to a well-defined and associative vertical composition in the 2-term case -- thus obtaining the Leibniz counterpart of the strict 2-category of 2-term Lie infinity algebras \cite{BC04}, see Subsection \ref{KanHomComp}, Theorem \ref{KanHom1} and Theorem \ref{KomComp2}.\medskip

Eventually, we provide, in Section 6, the definitions of the strict 2-category of Leibniz 2-algebras, which is 2-equivalent to the preceding 2-category.\medskip

An $\infty$-category structure on the class of $\infty$-algebras over a quadratic Koszul operad is being investigated independently of \cite{Get09} in a separate paper.


\section{Category of Leibniz infinity algebras}

Let $P$ be a quadratic Koszul operad. Surprisingly enough, $P_{\infty}$-structures on a graded vector space $V$ (over a field $\K$ of characteristic zero), which are essentially sequences $\ell_n$ of $n$-ary brackets on $V$ that satisfy a sequence $R_n$ of defining relations, $n\in\{1,2,\ldots\}$, are 1:1 \cite{GK94} with codifferentials
\be\label{GKCoalg}D\in\op{CoDer}^{1}({\cal F}^{\op{gr,c}}_{P^{\ac}}(s^{-1}V))\quad (|\ell_n|=2-n)\quad \text{
or }\quad D\in\op{CoDer}^{-1}({\cal F}^{\op{gr,c}}_{P^{\ac}}(sV))\quad (|\ell_n|=n-2)\;,\ee or, also, (if $V$ is finite-dimensional) 1:1 with differentials \be\label{GKAlgFin}d\in
\op{Der}^1({\cal F}^{\op{gr}}_{P^{!}}(sV^*))\quad(|\ell_n|=2-n)\quad \text{ or }\quad d\in \op{Der}^{-1}({\cal F}^{\op{gr}}_{P^{!}}(s^{-1}V^*))\quad(|\ell_n|=n-2)\;.\ee Here $\op{Der}^1({\cal F}^{\op{gr}}_{P^{!}}(sV^*))$ (resp., $\op{CoDer}^{1}({\cal F}^{\op{gr,c}}_{P^{\ac}}(s^{-1}V))$), for instance, denotes the space of endomorphisms of the free graded algebra over the Koszul dual operad $P^{!}$ of $P$ on the suspended linear dual $sV^*$ of $V$, which have degree 1 (with respect to the grading of the free algebra that is induced by the grading of $V$) and are derivations for each binary operation in $P^{!}$ (resp., the space of endomorphisms of the free graded coalgebra over the Koszul dual cooperad $P^{\ac}$ on the desuspended space $s^{-1}V$ that are coderivations) (by differential and codifferential we mean of course a derivation or coderivation that squares to 0).\medskip

Although the original nature of homotopified or oidified algebraic objects is their manifestation in terms of brackets \cite{BP12}, the preceding coalgebraic and algebraic settings are the most convenient contexts to think about such higher structures.

\subsection{Zinbiel (co)algebras}

\newcommand{\id}{\op{id}}

Since we take an interest mainly in the case where $P$ is the operad $\text{\sf Lei}$ (resp., the operad $\text{\sf Lie}$) of Leibniz (resp., Lie) algebras, the Koszul dual $P^{!}$ to consider is the operad $\text{\sf Zin}$ (resp., $\text{\sf Com}$) of Zinbiel (resp., commutative) algebras. We now recall the relevant definitions and results.

\begin{defi} A {\em graded Zinbiel algebra} $(${\em\small GZA}$)$ $($resp., {\em graded Zinbiel coalgebra} $($\em{\small GZC}$)$$)$ is a $\mathds{Z}$-graded vector space $V$ endowed with a multiplication, i.e. a degree 0 linear map $m:V\0 V\to V$ $($resp., a comultiplication, i.e. a degree 0 linear map $\zD:V\to V\0 V$$)$ that verifies the relation
\begin{equation}\label{definitionOfGZA}
m(\id\0 m)=m(m\0 \id)+m(m\0 \id)(\zt\0 \id)\quad (\text{resp.,}\; (\id\0 \zD)\zD=(\zD\0\id)\zD+(\zt\0\id)(\zD\0\id)\zD)\;,
\end{equation} where $\zt:V\0 V\ni u\0 v\mapsto (-1)^{|u||v|}v\0 u\in V\0 V$.
\end{defi}

When evaluated on homogeneous vectors $u,v,w\in V$, the Zinbiel relation for the multiplication $m(u,v)=:u\cdot v$ reads, $$ u\cdot(v\cdot w)=(u\cdot v)\cdot w+(-1)^{|u||v|}(v\cdot u)\cdot w\;.$$

\begin{ex}\label{propFreeZinbielProducFormulaOnTV}
The multiplication $\cdot$ on the reduced tensor module $\overline{T}(V):=\oplus_{n\ge 1} V^{\0 n}$ over a $\Z$-graded vector space $V$, defined, for homogeneous $v_i\in V$, by
\begin{equation}\label{FreeZinbielProducFormulaOnTV}
\begin{array}{l}
(v_1...v_p)\cdot(v_{p+1}...v_{p+q})=\sum\limits_{\sigma\in \op{Sh}(p,q-1)}(\sigma^{-1}\otimes \id)(v_1...v_{p+q})=\\
=\sum\limits_{\sigma\in \op{Sh}(p,q-1)}\varepsilon(\sigma^{-1})v_{\sigma^{-1}(1)}v_{\sigma^{-1}(2)}... v_{\sigma^{-1}(p+q-1)}v_{p+q}\;,\\
\end{array}\;
\end{equation}
where we wrote tensor products of vectors by simple juxtaposition, where $\op{Sh}(p,q-1)$ is the set of $(p,q-1)$-shuffles, and where $\ze(\zs^{-1})$ is the Koszul sign, endows $\overline{T}(V)$ with a {\em\small GZA} structure.\medskip

Similarly, the comultiplication $\zD$ on $\overline{T}(V)$, defined, for homogeneous $v_i\in V$, by
\begin{equation}\label{FreeZinbielCoProducFormulaOnTV}
\zD(v_1...v_p)=\sum_{k=1}^{p-1}\sum\limits_{\sigma\in \op{Sh}(k,p-k-1)}\varepsilon(\sigma)\lp v_{\sigma(1)}... v_{\sigma(k)}\rp\bigotimes\lp v_{\sigma(k+1)}... v_{\sigma(p-k-1)}v_{p}\rp\;,
\end{equation} is a {\em\small GZC} structure on $\overline{T}(V)$.
\end{ex}
\noindent As for the {\small GZA} multiplication on $\overline{T}(V),$ we have in particular
$$ v_1\cdot v_2=v_1 v_2\;;\quad (v_1v_2)\cdot v_3=v_1v_2v_3\;;$$ $$\quad v_1\cdot(v_2v_3)=v_1v_2v_3+(-1)^{|v_1||v_2|}v_2v_1v_3;\;\quad
(((v_1\cdot v_2)\cdot v_3)...)\cdot v_k=v_1v_2...v_k\;.$$

\begin{prop}
 The {\em\small GZA} $(\overline{T}(V),\cdot)$ $($resp., the {\em\small GZC} $(\overline{T}(V),\zD)$$)$ defined in Example~\ref{propFreeZinbielProducFormulaOnTV} is the {{\em free} {\small \em GZA} $($resp., {\em free} {\small \em GZC}$)$ over $V$}. We will denote it by $\op{Zin}(V)$ $($resp., $\op{Zin}^{\op{c}}(V)$$)$.
\end{prop}

\begin{defi} A {\em differential graded Zinbiel algebra} $(${\em\small DGZA}$)$ $($resp., a {\em differential graded Zinbiel coalgebra}$)$ $(${\em\small DGZC}$)$ is a {\em\small GZA} $(V,m)$ $($resp., {\em\small GZC} $(V,\zD)$$)$ together with a degree $1$ $($$-1$ in the homological setting$)$ derivation $d$ $($resp., coderivation $D$$)$ that squares to 0. More precisely, $d$ $($resp., $D$$)$ is a degree $1$ $($$-1$ in the homological setting$)$ linear map $d:V\to V$ $($resp., $D:V\to V$$)$, such that $$d\,m=m\lp d\0 \id+\id\0 d\rp\quad (\text{resp.,}\quad \zD\,D=\lp D\0\id +\id\0 D\rp\zD)\;$$ and $d^2=0$ $($resp., $D^2=0$$)$.
\end{defi}

Since the {\small GZA} $\op{Zin}(V)$ (resp., {\small GZC} $\op{Zin}^{\op{c}}(V)$) is free, any degree $1$ linear map $d:V\rightarrow \op{Zin}(V)$ (resp., $D:\op{Zin}^{\op{c}}(V)\to V$) uniquely extends to a derivation ${d}:\op{Zin}(V)\to\op{Zin}(V)$ (resp., coderivation $D:\op{Zin}^{\op{c}}(V)\to\op{Zin}^{\op{c}}(V)$).

\begin{defi} A {\em quasi-free} {\small DGZA} (resp., a {\em quasi-free} {\small DGZC}) over $V$ is a {\small DGZA} (resp., {\small DGZC}) of the type $(\op{Zin}(V),d)$ $($resp., $(\op{Zin}^{\op{c}}(V),D)\,$$)$.\end{defi}

\subsection{Leibniz infinity algebras}

In the present text we use homological ($i$-ary map of degree $i-2$) and cohomological ($i$-ary map of degree $2-i$) infinity algebras. Let us recall the definition of homological Leibniz infinity algebras.

\begin{defi}\label{LeibInftyAlg} A (homological) {\em Leibniz infinity algebra} is a graded vector space $V$ together with a sequence of linear maps $l_i: V^{\otimes i}\to V$ of degree $i-2$, $i\ge 1$, such that for any $n\ge 1$, the following {\em higher Jacobi identity} holds:
\begin{equation}\label{LodayInfinityAlgebraIdenities}
\begin{array}{l}
\displaystyle\sum\limits_{i+j=n+1}\sum\limits_{\substack{j\leqslant k\leqslant n}}\sum\limits_{\sigma\in \op{Sh}(k-j,j-1)}(-1)^{(n-k+1)(j-1)}\,(-1)^{j(v_{\sigma(1)}+...+v_{\sigma(k-j)})}\,\varepsilon(\sigma)\, \mathrm{sign}(\sigma) \\[4ex]\quad\quad
{l_i(v_{\sigma(1)},...,v_{\sigma(k-j)},l_j(v_{\sigma(k-j+1)},...,v_{\sigma(k-1)},v_k),v_{k+1},...,v_{n})}=0\;,
\end{array}
\end{equation} where $\op{sign}{\zs}$ is the signature of $\zs$ and where we denoted the degree of the homogeneous $v_i\in V$ by $v_i$ instead of $|v_i|$.
\end{defi}

\begin{theo}\label{LeibInftyAlg1:1} There is a 1:1 correspondence between Leibniz infinity algebras, in the sense of Definition \ref{LeibInftyAlg}, over a graded vector space $V$ and quasi-free {\em\small DGZC}-s $(\op{Zin}^{\op{c}}(sV),D)$ $($resp., in the case of a finite-dimensional graded vector space $V$, quasi-free {\small\em DGZA}-s $(\op{Zin}(s^{-1}V^*),d)$$)$.
\end{theo}

In the abovementioned 1:1 correspondence between infinity algebras over a quadratic Koszul operad $P$ and quasi-free {\small DG$P^{\ac}$C} (resp., quasi-free {\small DG$P^!$A}) (self-explaining notation), a $P_{\infty}$-algebra structure on a graded vector space $V$ is viewed as a representation on $V$ of the {\small DG} operad $P_{\infty}$ -- which is defined as the cobar construction $\zW P^{\ac}$ of the Koszul dual cooperad $P^{\ac}$. Theorem \ref{LeibInftyAlg1:1} makes this correspondence concrete in the case $P={\sf Lei}$; a proof can be found in \cite{AP10}.

\subsection{Leibniz infinity morphisms}

\begin{defi}\label{LeibInftyAlgMorph}
A {\em morphism between Leibniz infinity algebras} $(V,l_i)$ and $(W,m_i)$ is a sequence of linear maps $\varphi_i:V^{\otimes i}\to W$ of degree $i-1$, $i\ge 1$, which satisfy, for any $n\ge 1,$ the condition

\begin{equation}\label{LodayInfinityAlgebraMorphismIdenities}
\begin{array}{l}
\sum\limits_{i=1}^n\hspace{1mm}\sum\limits_{\substack{k_1+...+k_i=n}}\hspace{1mm}\sum\limits_{\sigma\in \op{\frak{Sh}}(k_1,...,k_i)}(-1)^{\sum\limits_{r=1}^{i-1}(i-r)k_r+\frac{i(i-1)}{2}}\,(-1)^{
\sum\limits_{r=2}^i(k_r-1)(v_{\sigma(1)}+...+v_{\sigma(k_1+...+k_{r-1})})}\,\varepsilon(\sigma)\,\mathrm{sign}(\sigma)\\[0.5cm]
\;m_i\left(\varphi_{k_1}(v_{\sigma(1)},...,v_{\sigma(k_1)}),\varphi_{k_2}(v_{\sigma(k_1+1)},...,v_{\sigma(k_1+k_2)}),...,\varphi_{k_i}(v_{\sigma(k_1+...+k_{i-1}+1)},...,v_{\sigma(k_1+...+k_i)})\right)\\[0,5cm]
=\\\quad\quad\quad\sum\limits_{i+j=n+1}\hspace{1mm}\sum\limits_{j\leqslant k \leqslant n}\sum\limits_{\sigma\in \op{Sh}(k-j,j-1)}(-1)^{k+(n-k+1)j}\,(-1)^{j(v_{\sigma(1)}+...+v_{\sigma(k-j)})}\,\varepsilon(\sigma)\,\mathrm{sign}(\sigma)\\[0.5cm]
\quad\quad\quad\quad\quad\quad{\varphi_i(v_{\sigma(1)},...,v_{\sigma(k-j)},l_j(v_{\sigma(k-j+1)},...,v_{\sigma(k-1)},v_k),v_{k+1},...,v_{n})}\;,
\end{array}
\end{equation} where $\op{\frak S\frak h}(k_1,\ldots,k_i)$ denotes the set of shuffles $\zs\in\op{Sh}(k_1,\ldots,k_i)$, such that $\zs(k_1)<\zs(k_1+k_2)<\ldots<\zs(k_1+k_2+\ldots+k_i)$.
\end{defi}

\begin{theo}\label{LeibInftyAlgMorph1:1}
There is a 1:1 correspondence between Leibniz infinity algebra morphisms from $(V,l_i)$ to $(W,m_i)$ and {\small \em DGC} morphisms $\mathrm{Zin}^{\op{c}}(sV)\to \mathrm{Zin}^{\op{c}}(sW)$ $($resp., in the finite-dimensional case, {\em\small DGA} morphisms $\mathrm{Zin}(s^{-1}W^*)\to \mathrm{Zin}(s^{-1}V^*)$$)$, where the quasi-free {\small\em DGZC}-s $($resp., the quasi-free {\small\em DGZA}-s$)$ are endowed with the codifferentials $($resp., differentials$)$ that encode the structure maps $l_i$ and $m_i$.
\end{theo}

In literature, infinity morphisms of $P_{\infty}$-algebras are usually defined as morphisms of quasi-free {\small DG$P^{\ac}$C}-s. However, no explicit formulae seem to exist for the Leibniz case. A proof of Theorem \ref{LeibInftyAlgMorph1:1} can be found in the first author's thesis. Let us also stress that the concept of infinity morphism of $P_{\infty}$-algebras does not coincide with the notion of morphism of algebras over the operad $P_{\infty}$.

\subsection{Composition of Leibniz infinity morphisms}\label{LeibInftyMorphComp}

Composition of infinity morphisms between $P_{\infty}$-algebras corresponds to composition of the corresponding morphisms between quasi-free {\small DG$P^{\ac}$C}-s: the categories {\tt $P_{\infty}$-Alg} and {\tt qfDG$P^{\ac}$CoAlg} (self-explaining notation) are isomorphic. Explicit formulae can easily be computed.

\section{Leibniz infinity homotopies}

\subsection{Concordances and their compositions}

Let us first look for a proper concept of homotopy in the category {\tt qfDG$P^{\ac}$CoAlg}, or, dually, in {\tt qfDG$P^{\,!}$Alg}.

\subsubsection{Definition and characterization}

The following concept of homotopy -- referred to as concordance -- first appeared in an unpublished work by Stasheff and Schlessinger, which was based on ideas of Bousfield and Gugenheim. It can also be found in \cite{SSS07}, for homotopy algebras over the operad {\sf Lie} (algebraic version), as well as in \cite{DP12}, for homotopy algebras over an arbitrary operad $P$ (coalgebraic version).\medskip

It is well-known that a $\Ci$-homotopy $\eta:I\times X\rightarrow Y$, $I=[0,1]$, connecting two smooth maps $p,q$ between two smooth manifolds $X,Y$, induces a cochain homotopy between the pullbacks $p^*,q^*.$ Indeed, in the algebraic category, $$
\eta^*:\Omega(Y)\rightarrow \Omega(I)\0\Omega(X)\;,
$$
and $\eta^*(\zw)$, $\zw\in\zW(Y),$ reads \be\label{ConDecomp}\eta^*(\zw)(t)=\varphi(\zw)(t)+dt\,\rho(\zw)(t)\;.\ee It is easily checked (see below for a similar computation) that, since $\eta^*$ is a cochain map, we have
$$
{d_t\varphi}=d_X\rho(t)+\rho(t) d_Y\;,
$$
where $d_X,d_Y$ are the de Rham differentials. When integrating over $I$, we thus obtain
$$
q^*-p^*=d_Xh+hd_Y\;,
$$
where $\displaystyle h=\smallint_I \rho(t) dt$ has degree $-1$.

\newcommand{\fg}{\frak g}
\newcommand{\fh}{\frak h}

Before developing a similar approach to homotopies between morphisms of quasi-free {\small DGZA}-s, let us recall that tensoring an `algebra' (resp., `coalgebra') with a {\small DGCA} (resp., {\small DGCC}) does not change the considered type of algebra (resp., coalgebra); let us also introduce the `evaluation' maps $$\ze_1^i:\zW(I)=\Ci(I)\oplus dt\,\Ci(I)\ni f(t)+dt\,g(t)\mapsto f(i)\in\K,\quad i\in\{0,1\}\;.$$

In the following -- in contrast with our above notation -- we omit stars. Moreover -- although the `algebraic' counterpart of a Leibniz infinity algebra over $V$ is ($\op{Zin}(s^{-1}V^*),d_V)$ -- we consider Zinbiel algebras of the type $(\op{Zin}(V),d_V)$.

\newcommand{\ceg}{\op{CE}({\frak g})}
\newcommand{\ceh}{\op{CE}({\frak h})}

\begin{defi}If $p,q:\op{Zin}(W)\to \op{Zin}(V)$ are two {\em\small DGA} morphisms, a {\em homotopy} or {\em concordance} $\zh:p\Rightarrow q$ from $p$ to $q$ is a {\em\small DGA} morphism
$\zh:\op{Zin}(W)\to \zW(I)\0\op{Zin}(V)$, such that $$\ze_1^0\zh=p\quad\text{and}\quad \ze_1^1\zh=q\;.$$\end{defi}

The following proposition is basic.

\newcommand{\beas}{\begin{eqnarray*}}
\newcommand{\eeas}{\end{eqnarray*}}

\begin{prop}\label{CharConcord} Concordances
$$\zh:\op{Zin}(W)\to\zW(I)\0\op{Zin}(V)$$ between {\em\small DGA} morphisms $p$, $q$ can be identified with 1-parameter families
$$\zvf:I\to\op{Hom}_{\op{DGA}}(\op{Zin}(W),\op{Zin}(V))\;$$and $$\zr:I\to \zvf\!\op{Der}(\op{Zin}(W),\op{Zin}(V))\;$$of (degree 0) {\em\small DGA} morphisms and of degree $1$ $\zvf$-Leibniz
morphisms, respectively, such that
{\be\label{DECon}d_t\zvf=[d,\zr(t)]\;\ee}and $\zvf(0)=p$, $\zvf(1)=q$. The {\em\small RHS} of the differential equation (\ref{DECon}) is defined by $$[d,\zr(t)]:=d_V\zr(t)+\zr(t)d_W\;,$$ where $d_V,d_W$ are the differentials of the quasi-free {\em\small DGZA}-s $\op{Zin}(V),\op{Zin}(W)$. \end{prop}

The notion of $\zvf$-derivation or $\zvf$-Leibniz morphism appeared for instance in \cite{BKS04}: for $w,w'\in\op{Zin}(W)$, $w$ homogeneous, $$\zr(w\cdot w')=\zr(w)\cdot\zvf(w')+(-1)^{w}\zvf(w)\cdot\zr(w')\;,$$ where we omitted the dependence of $\zr$ on $t$.

\begin{proof} As already mentioned in Equation (\ref{ConDecomp}), $\zh(w)$, $w\in\op{Zin}(W)$, reads $$\zh(w)(t)=\zvf(w)(t)+dt\,\zr(w)(t)\;,$$ where $\zvf(t):\op{Zin}(W)\to\op{Zin}(V)$ and $\zr(t):\op{Zin}(W)\to\op{Zin}(V)$ have degrees $0$ and $1$, respectively (the grading of $\op{Zin}(V)$ is induced by that of $V$ and the grading of $\zW(I)$ is the homological one). Let us now translate the remaining properties
of $\zh$ into properties of $\zvf$ and $\zr$. We denote by $d_I=dt\,d_t$ the de Rham differential of $I$. Since $\zh$ is a chain map,
$$dt\,d_t\zvf+d_V\zvf-dt\,d_V\zr=(d_I\0\id+\id\0 d_V)\zh=\zh d_W=\zvf d_W+dt\, \zr d_W\;,$$ so that $$d_{V}\zvf=\zvf d_{W}\quad \text{and}\quad d_t\zvf=d_{V}\zr+\zr d_{W}=[d,\zr]\;.$$ As $\zh$
is also an algebra morphism, we have, for $w,w'\in\op{Zin}(W),$ $$
\zvf(w\cdot w')+dt\,\zr(w\cdot w')=
(\zvf(w)+dt\,\zr(w))\cdot(\zvf(w')+dt\,\zr(w')) $$ $$= \zvf(w)\cdot\zvf(w')+(-1)^{w}dt\,(\zvf(w)\cdot\zr(w')) +dt\,(\zr(w)\cdot\zvf(w'))\;,$$ and $\zvf$ (resp., $\zr$) is a family of {\small DGA} morphisms (resp., of degree $1$ $\zvf$-Leibniz maps) from $\op{Zin}(W)$ to $\op{Zin}(V)$. Eventually, $$p=\ze_1^0\zh=\zvf(0)\quad\text{and}\quad q=\ze_1^1\zh=\zvf(1)\;.$$\end{proof}

\newcommand{\mk}{\mathds{k}}

\subsubsection{Horizontal and vertical compositions}

{\em In literature, the `categories' of Leibniz (resp., Lie) infinity algebras over $V$ (finite-dimensional) and of quasi-free {\small DGZA}-s (resp., quasi-free {\small DGCA}-s) over $s^{-1}V^*$ are (implicitly or explicitly) considered equivalent}. This conjecture is so far corroborated by the results of this paper. Hence, let us briefly report on compositions of concordances.\medskip

Let $\zh:p\Rightarrow q$ and $\zh':p'\Rightarrow q'$, 

\begin{equation}\label{horizontal_composition_diagramm}
\xymatrix
@R=1pc{
&\ar@{=>}^{\eta}[dd]&&\ar@{=>}^{\eta'}[dd]&\\
(\mathrm{Zin}(W),d_W)\ar@/^2pc/@{->}[rr]^-{p}\ar@/_2pc/@{->}[rr]_-{q}&&(\mathrm{Zin}(V),d_V)\ar@/_2pc/@{->}[rr]_-{q'}\ar@/^2pc/@{->}[rr]^-{p'}&&(\mathrm{Zin}(U),d_U)\;,\\
&&&&
}
\end{equation}
be concordances between {\small DGA} morphisms. Their horizontal composite $\zh'\circ_0\zh:p'\circ p\Rightarrow q'\circ q$,
$$
\xymatrix
@R=1pc{
&\ar@{=>}^{\eta'\circ_0\eta}[dd]\\
(\mathrm{Zin}(W),d_W)\ar@/^2pc/@{->}[rr]^-{p'\circ p}\ar@/_2pc/@{->}[rr]_-{q'\circ q}&&(\mathrm{Zin}(U),d_U)\;,\\
&&&&
}
$$
is defined by
\begin{equation}
(\eta'\circ_0\eta)(t)=(\varphi'(t)\circ\varphi(t))+dt\,(\varphi'(t)\circ\rho(t)+\rho'(t)\circ\varphi(t))\;,
\end{equation}
with self-explaining notation. It is easily checked that the first term of the {\small RHS} and the coefficient of $dt$ in the second term have the properties needed to make $\zh'\circ_0\zh$ a concordance between $p'\circ p$ and $q'\circ q$.

As for the vertical composite $\zh'\circ_1\zh:p\Rightarrow r$ of concordances $\zh:p\Rightarrow q$ and $\zh':q\Rightarrow r$,

$$\xymatrix@C=3pc@R=3pc
{
&\ar@{=>}[d]^{\eta}&\\
(\mathrm{Zin}(W),d_W)\ar@/^4pc/@{->}^-p[rr]\ar@{->}^-(0.43)q[rr]\ar@/_4pc/@{->}_-r[rr]&\ar@{=>}^{\zh'}[d]&(\mathrm{Zin}(V),d_V)\;,\\
&&
}$$ note that the composability condition $\zvf(1)=q=t(\zh)=s(\zh')=q=\zvf'(0)$, where $s,t$ denote the source and target maps, does not encode any information about $\zr(1),\zr'(0)$. Hence, the usual `half-time' composition cannot be applied.

\begin{rem} The preceding observation is actually the shadow of the fact that the `category' of Leibniz infinity algebras is an infinity category.\end{rem}

\subsection{Infinity homotopies}\label{sec_Getzler_Shoiket}

Some authors addressed directly or indirectly the concept of homotopy of Lie infinity algebras ($L_{\infty}$-algebras). As mentioned above, in the (equivalent) `category' of quasi-free {\small DGCA}-s, the classical picture of homotopy leads to concordances. In the `category' of $L_{\infty}$-algebras itself, morphisms can be viewed as Maurer-Cartan ({\small MC}) elements of a specific $L_{\infty}$-algebra \cite{Dol07},\cite{BS07}, which yields the notion of `gauge homotopy' between $L_{\infty}$-morphisms. Additional notions of homotopy between {\small MC} elements do exist: Quillen and cylinder homotopies. On the other hand, Markl \cite{Mar02} uses colored operads to construct homotopies for $\infty$-morphisms in a systematic way. The concepts of concordance, operadic homotopy, as well as Quillen, gauge, and cylinder homotopies are studied in detail in \cite{DP12}, for homotopy algebras over any Koszul operad, and they are shown to be equivalent, essentially due to homotopy transfer.\medskip

In this subsection, we focus on the Leibniz infinity case and provide a brief account on the relationship between concordances, gauge homotopies, and Quillen homotopies (in the next section, we explain why the latter concept is the bridge to Getzler's \cite{Get09} (and Henriques' \cite{henriques}) work, as well as to the infinity category structure on the set of Leibniz infinity algebras).\medskip

Let us stress that all series in this section converge under some local finiteness or nilpotency conditions (for instance pronilpotency or completeness).

\subsubsection{Gauge homotopic Maurer-Cartan elements}\label{Gauge}

Lie infinity algebras over ${\frak g}$ are in bijective correspondence with quasi-free {\small DGCC}-s $(\op{Com}^c(s{\frak g}),D)$, see Equation (\ref{GKCoalg}). Depending on the definition of the $i$-ary brackets $\ell_i$, $i\ge 1$, from the corestrictions $D_i:(s{\frak g})^{\odot i}\to s{\frak g}$, where $\odot$ denotes the graded symmetric tensor product, one obtains various sign conventions in the defining relations of a Lie infinity algebra. When setting $\ell_i:=D_i$ (resp., $\ell_i:=s^{-1}D_i\, s^{\,i}$ (our choice in this paper), $\ell_i:=(-1)^{i(i-1)/2}s^{-1}D_i\, s^{\,i}$), we get a Voronov $L_{\infty}$-antialgebra \cite{Vor05} (resp., a Lada-Stasheff $L_{\infty}$-algebra \cite{LS93}, a Getzler $L_{\infty}$-algebra \cite{Get09}) made up of graded symmetric multilinear maps $\ell_i:(s{\frak g})^{\times i}\to s{\frak g}$ of degree $-1$ (resp., by graded antisymmetric multilinear maps $\ell_i:{\frak g}^{\times i}\to {\frak g}$ of degree $i-2$, idem), which verify the conditions
\be\sum_{i+j=r+1}\sum_{\zs\in\op{Sh}(i,j-1)}\ze(\zs)\ell_{j}(\ell_i(sv_{\zs_1},\ldots,sv_{\zs_i}),sv_{\zs_{i+1}},\ldots,sv_{\zs_r})=0\;,\label{VLInfty}\ee for all homogeneous $sv_k\in s{\frak g}$ and all $r\ge 1$ (resp., the same higher Jacobi identities (\ref{VLInfty}), except that the sign $\ze(\zs)$ is replaced by $(-1)^{i(j-1)}\ze(\zs)\op{sign}(\zs)$, by $(-1)^{i}\ze(\zs)\op{sign}(\zs)$).\medskip

As the {\small MC} equation of a Lie infinity algebra $(\frak g,\ell_i)$ must correspond to the {\small MC} equation given by the $D_i$, it depends on the definition of the operations $\ell_i$. For a Lada-Stasheff $L_{\infty}$-algebra (resp., a Getzler $L_{\infty}$-algebra), we obtain that the set $\op{MC}(\frak g)$ of {\small MC} elements of $\frak g$ is the set of solutions $\za\in\frak g_{-1}$ of the {\small MC} equation \be\label{MCLInfty}{\sum_{i=1}^{\infty}\frac{1}{i!}(-1)^{i(i-1)/2}\ell_i(\za,\ldots,\za)=0}\quad (\text{resp.}, \sum_{i=1}^{\infty}\frac{1}{i!}\ell_i(\za,\ldots,\za)=0)\;.\ee

\begin{rem} Since we prefer the original definition of homotopy Lie \cite{LS93} and homotopy Leibniz \cite{AP10} algebras, but use the results of \cite{Get09}, we adopt -- to facilitate the comparison with \cite{Get09} -- Getzler's sign convention whenever his work is involved, and change signs appropriately later, when applying the results in our context.\end{rem}

Hence, we now consider the second {\small MC} equation (\ref{MCLInfty}). Further, for any $\za\in\frak g_{-1}$, the twisted brackets {$$\ell^{\za}_i(v_1,\ldots,v_i)=\sum_{k=0}^{\infty}\frac{1}{k!}\ell_{k+i}(\za^{\0
k},v_1,\ldots,v_i)\;,$$} $v_1,\ldots,v_i\in {\frak g}$, are a sequence of graded antisymmetric multilinear maps of degree $i-2$. It is well-known that the $\ell_i^{\za}$ endow $\frak g$ with a new Lie infinity structure, if $\za\in\op{MC}({\frak g}).$ Finally, any vector $r\in{\frak g}_0$ gives rise to a vector field \be V_r:{\frak g}_{-1}\ni\za\mapsto {V_r(\za):=-\ell_1^{\za}(r)}=-\sum_{k=0}^{\infty}\frac{1}{k!}\ell_{k+1}(\za^{\0
k},r)=\sum_{k=1}^{\infty}\frac{(-1)^{k}}{(k-1)!}\ell_{k}(r,\za^{\0 (k-1)})\in{\frak g}_{-1}\;.\label{SpecVectfield}\ee This field restricts to a vector field of the Maurer-Cartan quadric $\op{MC}(\frak g)$ \cite{DP12}. It follows that the integral curves {\be\label{FlowEq}d_t\za=V_r(\za(t))\;,\ee} starting from points in the quadric, are located inside $\op{MC}(\frak g)$. Hence, the

\begin{defi}\label{DolSho} $($\cite{Dol07}, \cite{BS07}$)$ Two {\small MC} elements $\za,\zb\in\op{MC}({\frak g})$ of a Lie infinity algebra $\frak g$ are {\em gauge homotopic} if there exists $r\in{\frak g}_0$ and an integral curve $\za(t)$ of $V_r$, such that $\za(0)=\za$ and $\za(1)=\zb$.\end{defi}

This gauge action is used to define the deformation functor $\op{Def}:{\tt L}_{\infty}\to{\tt Set}$ from the category of Lie infinity algebras to the category of sets. Moreover, it will provide a concept of homotopy between Leibniz infinity morphisms.\medskip

Let us first observe that Equation (\ref{FlowEq}) is a 1-variable ordinary differential equation ({\small ODE}) and can be solved via an iteration procedure. The integral curve with initial point $\za\in\op{MC}(\frak g)$ is computed in \cite{Get09}. When using our sign convention in the defining relations of a Lie infinity algebra, we get an {\small ODE} that contains different signs and the solution of the corresponding Cauchy problem reads \begin{equation} \alpha(t)=\sum_{k=0}^\infty\frac{t^k}{k!}e_{\za}^k(r)\;,\label{GetzlerGaugeHomotopyExplicit}\end{equation} where the $e_{\za}^k(r)$ admit a nice combinatorial description in terms of rooted trees. Moreover, they can be obtained inductively:
\begin{equation}\label{Getzler induction formula}
\left\{\begin{array}{l}
\displaystyle e_{\za}^{i+1}(r)=\sum\limits_{n=0}^{\infty}\frac{1}{n!}(-1)^{\frac{n(n+1)}{2}}\sum\limits_{k_1+...+k_n=i}\frac{i!}{k_1!...k_n!}\ell_{n+1}(e_{\za}^{k_1}(r),...,e_{\za}^{k_n}(r),r)\;,\\
e^0_{\za}(r)=\alpha\;.
\end{array}\right.
\end{equation}
It follows that $\za,\zb\in\op{MC}({\frak g})$ are gauge homotopic if \be\zb-\za=\sum_{k=1}^{\infty}\frac{1}{k!}e_{\za}^k(r)\;,\label{FundaHomotopyEquation}\ee for some $r\in{\frak g}_0.$

\subsubsection{Simplicial de Rham algebra}\label{SimpicialDR}

\newcommand*{\EnsQuot}[2]%
{\ensuremath{%
    #1/\!\raisebox{-.65ex}{\ensuremath{{#2}}}}}

We first fix the notation.\medskip

Let $\zD$ be the {\it simplicial category} with objects the nonempty finite ordinals $[n]=\{0,\ldots, n\}$, $n\ge 0$, and morphisms the order-respecting functions $f:[m]\to [n]$. Denote by $\zd_n^i:[n-1]\rightarrowtail [n]$ the injection that omits the image $i$ and by $\zs_n^i:[n+1]\twoheadrightarrow [n]$ the surjection that assigns the same image to $i$ and $i+1$, $i\in\{0,\ldots,n\}$.

A {\it simplicial object} in a category {\tt C} is a functor $X \in [\zD^{\op{op}},{\tt C}\,]\;.$ It is completely determined by the simplicial data $(X_n,d\,^n_i,s\,^n_i)$, ${n\ge 0,\, i\in\{0,\dots,n\}},$ where $X_n=X[n]$ ($n$-simplices), $d\,^n_i=X(\zd^i_n)$ (face maps), and $s\,^n_i=X(\zs^i_n)$ (degeneracy maps). We denote by ${\tt SC}$ the functor category $[\zD^{\op{op}},{\tt C}]$ of simplicial objects in ${\tt C}$.

The simplicial category is embedded in its Yoneda dual category: $$h_{\ast}: \zD\ni[n]\mapsto \h_{\zD}(-,[n])\in [\zD^{\op{op}},{\tt Set}]={\tt SSet}\;.$$ We refer to the functor of points of $[n]$, i.e. to the simplicial set $\zD[n]:=\h_{\zD}(-,[n]),$ as the {\it standard simplicial $n$-simplex}. Moreover, the Yoneda lemma states that $$\h_{\zD}([n],[m])\simeq \h(\h_{\zD}(-,[n]),\h_{\zD}(-,[m]))=\h(\zD[n],\zD[m]) \;.$$ This bijection sends $f:[n]\to [m]$ to $\zvf$ defined by $\zvf_{[k]}(\bullet)= f\circ \bullet$ and $\zvf$ to $\zvf_{[n]}(\op{id}_{[n]})$. In the following we identify $[n]$ (resp., $f$) with $\zD[n]$ (resp., $\zvf$).

The set $S_n$ of $n$-simplices of a simplicial set $S$ is obviously given by $S_n\simeq\h(\h_{\zD}(-,[n]),S)=\h(\zD[n],S).$

Let us also recall the adjunction $$|-|:{\tt SSet}\rightleftarrows {\tt Top}:\op{Sing}$$ given by the `geometric realization functor' $|-|$ and the `singular complex functor' Sing. To define $|-|$, we first define the realization $|\zD[n]|$ of the standard simplicial $n$-simplex to be the {\it standard topological $n$-simplex} $$\zD^n=\{(x_0,\ldots,x_n)\in \R^{{n+1}}: x_i\ge 0, \sum_ix_i=1\}\;.$$ We can view $|-|$ as a functor $|-|\in[\zD,{\tt Top}]\;$. Indeed, if $f:[n]\to [m]$ is an order-respecting map, we can define a continuous map $|f|:\zD^n\to \zD^m$ by $$|f|[x_0,\ldots,x_n]=[y_0,\ldots,y_m]\;,$$ where $y_i=\sum_{j\in f^{-1}\{i\}}x_j$.\medskip

Let $\w_{\K}^{\star}(x_0,\ldots,x_n,dx_0,\ldots,dx_n)$ be the free graded commutative algebra generated over $\K$ by the degree 0 (resp., degree 1) generators $x_i$ (resp., $dx_i$). If we divide out the relations $\sum_ix_i=1$ and $\sum_idx_i=0$ and set $d(x_i)=dx_i$ and $d(dx_i)=0$, we obtain a quotient {\small DGCA} $$\zW_n^{\star}=\EnsQuot{\w_{\K}^{\star}(x_0,\ldots,x_n,dx_0,\ldots,dx_n)}{\lp\sum_ix_i-1,\sum_idx_i\rp}$$ that can be identified, for $\K=\R$, with the algebra of polynomial differential forms $\zW^{\star}(\zD^n)$ of the standard topological $n$-simplex $\zD^n$. When defining $\zW^{\star}:\zD^{\op{op}}\to {\tt DGCA}$ by $\zW^{\star}[n]:=\zW^{\star}_n$ and, for $f:[n]\to [m]$, by $\zW^{\star}(f):=|f|^*:\zW^{\star}_m\to \zW^{\star}_n$ (use the standard pullback formula for differential forms given by $y_i=\sum_{j\in f^{-1}\{i\}}x_j$), we obtain a simplicial differential graded commutative algebra $\zW^{\star}\in {\tt SDGCA}$. Hence, the face maps $d\,^n_i:\zW^{\star}_n\to \zW^{\star}_{n-1}$ are the pullbacks by the $|\zd^i_n|:\zD^{n-1}\to \zD^n$, and similarly for the degeneracy maps. In particular, $d^2_0=|\zd_2^0|^*:\zW^{\star}_2\to \zW^{\star}_{1}$ is induced by $y_0=0,y_1=x_0,y_2=x_1$. Let us eventually introduce, for $0\le i\le n$, the vertex $e_i$ of $\zD^n$ whose $(i+1)$-th component is equal to $1$, as well as the evaluation map $\ze_n^i:\zW^{\star}_n\to \K$ at $e_i$ (`pullback' induced by $(y_0,\ldots,y_n)=e_i$).

\subsubsection{Quillen homotopic Maurer-Cartan elements}\label{QuillenDefSect}

We already mentioned that, if $({\frak g},\ell_i)$ is an $L_{\infty}$-algebra and $(A,\cdot,d)$ a {\small DGCA}, their tensor product ${\frak g}\0 A$ has a canonical $L_{\infty}$-structure $\bar\ell_i$. It is given by $$\bar\ell_1(v\0 a)=(\ell_1\0\op{id}+\op{id}\0 d)(v\0 a)=\ell_1(v)\0 a+(-1)^vv\0 d(a)$$ and, for $i\ge 2$, by $$\bar\ell_i(v_1\0 a_1,\ldots,v_i\0 a_i)=\pm\ell_i(v_1,\ldots,v_i)\0(a_1\cdot \ldots \cdot a_i)\;,$$ where $\pm$ is the Koszul sign generated by the commutation of the variables.\medskip

The following concept originates in Rational Homotopy Theory.

\begin{defi} Two {\small MC} elements $\za,\zb\in\op{MC}({\frak g})$ of a Lie infinity algebra $\frak g$ are {\em Quillen homotopic} if there exists a {\small MC} element $\bar\zg\in\op{MC}(\bar{\frak g}_1)$ of the Lie infinity algebra $\bar{\frak g}_1:={\frak g}\0\zW^{\star}_1$, such that $\ze_1^0\bar\zg=\za$ and $\ze_1^1\bar\zg=\zb$ $($where the $\ze_1^i$ are the natural extensions of the evaluation maps$)$.\end{defi}

From now on, we accept, in the definition of gauge equivalent {\small MC} elements, vector fields $V_{r(t)}$ induced by time-dependent $r=r(t)\in{\frak g}_0$. The next result is proved in \cite{Can99} (see also \cite{Man99}). A sketch of the proof will be given later.

\begin{prop}
Two {\small MC} elements of a Lie infinity algebra are Quillen homotopic if and only if they are gauge homotopic.
\end{prop}

\subsubsection{Infinity morphisms as Maurer-Cartan elements and infinity homotopies}

The possibility to view morphisms in $\h_{\op{DGP^{\ac}C}}(C,{\cal F}^{\op{gr,c}}_{P^{\ac}}(sW))$ as {\small MC} elements is known from the theory of the bar and cobar constructions of algebras over an operad. In \cite{DP12}, the authors showed that the fact that infinity morphisms between $P_{\infty}$-algebras $V$ and $W$, i.e. morphisms in $$\h_{\op{DGP^{\ac}C}}({\cal F}^{\op{gr,c}}_{P^{\ac}}(sV), {\cal F}^{\op{gr,c}}_{P^{\ac}}(sW))\;,$$ are 1:1 with Maurer-Cartan elements of an $L_{\infty}$-structure on $$\h_{\K}({\cal F}^{\op{gr,c}}_{P^{\ac}}(sV),W)\;,$$ is actually a consequence of a more general result based on the encoding of two $P_{\infty}$-algebras and an infinity morphism between them in a {\small DG} colored free operad. In the case $P={\sf Lie}$, one recovers the fact \cite{BS07} that \be\label{NatBijCoAlg}\h_{\op{DGCC}}(C,\op{Com}^c(sW))\simeq \op{MC}(\h_{\K}(C,W))\;,\ee where $C$ is any locally conilpotent {\small DGCC}, where $W$ is an $L_{\infty}$-algebra, and where the {\small RHS} is the set of {\small MC} elements of some convolution $L_{\infty}$-structure on $\h_{\K}(C,W)$.\medskip

In the sequel we detail the case $P={\sf Lei}$. Indeed, when interpreting infinity morphisms of Leibniz infinity algebras as {\small MC} elements of a Lie infinity algebra, the equivalent notions of gauge and Quillen homotopies provide a concept of homotopy between Leibniz infinity morphisms.

\begin{prop}\label{proposition_about_lie_inf_on_space_of_homomorphisms}
Let $(V,\ell_i)$ and $(W, m_i)$ be two Leibniz infinity algebras and let $(\mathrm{Zin}^c(sV),D)$ be the quasi-free {\small\em DGZC} that corresponds to $(V,\ell_i)$. The graded vector space $$L(V,W):=\h_{\K}(\mathrm{Zin}^c(sV),W)$$ carries a convolution Lie infinity structure given by
\begin{equation}\label{lieInfinityStr1}
\mathcal{L}_1f=m_1\circ f+(-1)^ff\circ D
\end{equation}
and, as for $\mathcal{L}_p(f_1,...,f_p)$, $p\ge 2$, by
\begin{equation}\label{lieInfinityStrp}
\xymatrix@C=3.5pc{
\mathrm{Zin}^c(sV)\ar@{->}^-{\Delta^{p-1}}[r]&(\mathrm{Zin}^c(sV))^{\otimes p}\ar@{->}^-{\sum\limits_{\sigma\in S(p)}\varepsilon(\sigma)\, \op{sign}(\sigma)\, f_{\sigma(1)}\otimes...\otimes f_{\sigma(p)}}[rr]&\qquad\quad&W^{\otimes p}\ar@{->}^-{m_p}[r]&W\;,
}
\end{equation}
where $f,f_1,\ldots,f_p\in L(V,W)$, $\Delta^{p-1}=(\Delta\otimes \op{id}^{\otimes(p-2)})...(\Delta\otimes \op{id})\Delta\,$, where $S(p)$ denotes the symmetric group on $p$ symbols, and where the central arrow is the graded antisymmetrization operator.
\end{prop}
\begin{proof}
See \cite{Laa02} or \cite{DP12}. A direct verification is possible as well.
\end{proof}

\begin{prop}\label{MorphMC}
Let $(V,\ell_i)$ and $(W, m_i)$ be two Leibniz infinity algebras. There exists a 1:1 correspondence between the set of infinity morphisms from $(V,\ell_i)$ to $(W,m_i)$ and the set of {\small MC} elements of the convolution Lie infinity algebra structure ${\cal L}_i$ on $L(V,W)$ defined in Proposition~\ref{proposition_about_lie_inf_on_space_of_homomorphisms}.
\end{prop}

Observe that the considered {\small MC} series converges pointwise. Indeed, the evaluation of $\mathcal{L}_p(f_1,...,f_p)$ on a tensor in $\op{Zin}^c(sV)$ vanishes for $p\gg$, in view of the local conilpotency of $\op{Zin}^c(sV)$. Moreover, convolution $L_{\infty}$-algebras are complete, so that their {\small MC} equation converges in the topology induced by the filtration (a descending filtration $F^iL$ of the space $L$ of an $L_{\infty}$-algebra $(L,{\cal L}_k)$ is \emph{compatible} with the $L_{\infty}$-structure ${\cal L}_k$, if ${\cal L}_k(F^{i_1}L,\ldots,F^{i_k}L)\subset F^{i_1+\ldots+i_k}L$, and it is \emph{complete}, if, in addition, the `universal' map $L\to  \varprojlim L/F^iL$ from $L$ to the (projective) limit of the inverse system $L/F^iL$ is an isomorphism) \cite{DP12}.\medskip

Note also that Proposition \ref{MorphMC} is a specification, in the case $P={\sf Lei}$, of the abovementioned 1:1 correspondence between infinity morphisms of $P_{\infty}$-algebras and {\small MC} elements of a convolution $L_{\infty}$-algebra. To increase the readability of this text, we give nevertheless a sketchy proof.

\begin{proof} An {\small MC} element is an $\alpha\in \h_{\K}(\mathrm{Zin}^c(sV),W)$ of degree $-1$ that verifies the {\small MC} equation. Hence, $s\za:\op{Zin}^c(sV)\to sW$ has degree 0 and, since $\op{Zin}^c(sW)$ is free as {\small GZC}, $s\za$ coextends uniquely to $\widehat{s\za}\in\h_{\op{GZC}}(\op{Zin}^c(sV),\op{Zin}^c(sW))$. The fact that $\za$ is a solution of the {\small MC} equation exactly means that $\widehat{s\za}$ is a {\small DGZC}-morphism, i.e. an infinity morphism between the Leibniz infinity algebras $V$ and $W$. Indeed, when using e.g. the relations $\ell_i= s^{-1} D_i\, s^{\,i}$ and $m_i=s^{-1} {\frak D}_i\,s^{\,i}$, and the corresponding version of the {\small MC} equation, we get
$$
\begin{array}{l}
\sum\limits^\infty_{p=1}\frac{1}{p!}(-1)^{\frac{p(p-1)}{2}}\mathcal{L}_p(\alpha,...,\alpha)=0 \Leftrightarrow\\
\sum\limits^\infty_{p=1} (-1)^{\frac{p(p-1)}{2}}m_p(\alpha\otimes...\otimes\alpha)\Delta^{p-1}+(-1)^\alpha\alpha D=0\Leftrightarrow\\
\sum\limits^\infty_{p=1}(-1)^{\frac{p(p-1)}{2}}s^{-1}{\frak D}_p\,s^{\, p}(\alpha\otimes...\otimes\alpha)\Delta^{p-1}+(-1)^\alpha\alpha D=0 \Leftrightarrow\\
\sum\limits^\infty_{p=1}s^{-1}{\frak D}_p(s\alpha\otimes...\otimes s\alpha)\Delta^{p-1}-s^{-1}s\alpha D=0 \Leftrightarrow\\
{\frak D}\widehat{(s\alpha)}=\widehat{(s\za)}D\;.
\end{array}
$$
\end{proof}

Hence, the

\begin{defi} Two infinity morphisms $f,g$ between Leibniz infinity algebras $(V,\ell_i)$, $(W,m_i)$ are {\em infinity homotopic}, if the corresponding {\small MC} elements $\za=\za(f)$ and $\zb=\zb(g)$ of the convolution Lie infinity structure ${\cal L}_i$ on $L=L(V,W)$ are Quillen (or gauge) homotopic. In other words, $f$ and $g$ are infinity homotopic, if there exists $\bar\zg\in\op{MC}_1(\bar{L})$, i.e. an {\small MC} element $\bar\zg$ of the Lie infinity structure $\bar{\cal L}_i$ on $\bar{L}=L\0\zW^{\star}_1$ -- obtained from the convolution structure ${\cal L}_i$ on $L=\h_{\K}(\op{Zin}^c(sV),W)$ via extension of scalars --, such that $\ze_1^0\bar\zg=\za$ and $\ze_1^1\bar\zg=\zb$.\end{defi}

\subsubsection{Comparison of concordances and infinity homotopies}

Since, according to the prevalent philosophy, the `categories' {\tt qfDG$P^{\ac}$CoAlg} and {\tt $P_{\infty}$-Alg} are `equivalent', appropriate concepts of homotopy in both categories should be in 1:1 correspondence. It can be shown \cite{DP12} that, for any type of algebras, the concepts of concordance and of Quillen homotopy are equivalent (at least if one defines concordances in an appropriate way); and as Quillen homotopies are already known to be equivalent to gauge homotopies, the desired result follows in whole generality.\medskip

To accommodate the reader who is not interested in (nice) technicalities, we provide now a sketchy explanation of both relationships, defining concordances dually and assuming for simplicity that $\K=\R$.\medskip

Remember first that we defined concordances, in conformity with the classical picture, in a contravariant way: two infinity morphisms $f,g:V\to W$ between homotopy Leibniz algebras, i.e. two {\small DGA} morphisms $f^*,g^*:\op{Zin}(s^{-1}W^*)\to \op{Zin}(s^{-1}V^*)$, are concordant if there is a morphism $$\zh\in\h_{\op{DGA}}(\op{Zin}(s^{-1}W^*),\op{Zin}(s^{-1}V^*)\0\zW^{\star}_1)\;,$$ whose values at $0$ and $1$ are equal to $f^*$ and $g^*$, respectively. Although we will use this definition in the sequel (observe that we slightly adapted it to future purposes), we temporarily prefer a dual, covariant definition (which has the advantage that the spaces $V,W$ need not be finite-dimensional).\medskip

The problem that the linear dual of the infinite-dimensional {\small DGCA} $\zW^{\star}_1$ (let us recall that $\star$ stands for the (co)homological degree) is not a coalgebra, has already been addressed in \cite{BM12}. The authors consider a space $(\zW_1^{\star})^{\vee}$ made up by the formal series, with coefficients in $\K$, of the elements $\alpha_i=(t^i)^\vee,$ $i\ge0$, and $\beta_i=(t^i\,dt)^\vee,$ $i\ge 0$. For instance, $\sum_{i\in\N}{\frak K}^i\za_i$ represents the map $\{t^i\}_{i\in\N}\to \K$ and assigns to each $t^i$ the coefficient ${\frak K}^i$. The differential $\p$ of $(\zW_1^{\star})^\vee$ is (dual to the de Rham differential $d$ and is) defined by $\p(\alpha_i)=0$ and $\p(\beta_i)=(i+1)\alpha_{i+1}$. As for the coalgebra structure $\zd$, we set
\begin{gather*}
\delta(\alpha_i)=\sum_{a+b=i}\alpha_a\otimes\alpha_b\;,\\
\delta(\beta_i)=\sum_{a+b=i}(\beta_a\otimes\alpha_b+\alpha_a\otimes\beta_b)\;.
\end{gather*}
When extending to all formal series, we obtain a map $\zd:(\zW_1^{\star})^\vee\to (\zW_1^\star)^\vee\widehat\0(\zW_1^\star)^\vee$, whose target is the completed tensor product. To fix this difficulty, one considers the decreasing sequence of vector spaces $(\zW_1^\star)^\vee=:\zL^0\supset \zL^1\supset\zL^2\supset\ldots\;$, where $\zL^i=\zd^{-1}(\zL^{i-1}\0\zL^{i-1})$, $i\ge 1$, and defines the universal {\small DGCC} $\zL:=\cap_{i\ge 1}\zL^i$.\medskip

A concordance can then be defined as a map $$\zh\in\h_{\op{DGC}}(\op{Zin}^c(sV)\0\zL,\op{Zin}^c(sW))$$ (with the appropriate boundary values). It is then easily seen that any Quillen homotopy, i.e. any element in $\op{MC}(L\0\zW_1^\star)$, gives rise to a concordance. Indeed, set ${\cal V}:=sV$ and note that $$L\0\zW_1^\star=\h_{\K}(\op{Zin}^c({\cal V}),W)\0\zW_1^\star=\h_{\K}\lp\bigoplus_{i\ge 1}{\cal V}^{\0 i},W\rp\0\zW_1^\star=\lp\prod_{i\ge 1}\h_{\K}({\cal V}^{\0 i},W)\rp\0\zW_1^\star$$ $$\longrightarrow\; \prod_{i\ge 1}\lp\h_{\K}({\cal V}^{\0 i},W)\0\zW_1^\star\rp\;\longrightarrow \; \prod_{i\ge 1}\h_{\K}({\cal V}^{\0 i}\0\zL,W)=\h_{\K}(\op{Zin}^c({\cal V})\0\zL,W)\;.$$ Only the second arrow is not entirely obvious. Let $(f,\zw)\in \h_{\K}({\cal V}^{\0 i},W)\times\zW_1^\star$ and let $(v,\zl)\in{\cal V}^{\0 i}\times \zL$. Set $\zw=\sum_{a=0}^{N}k_at^a+\sum_{b=0}^N\zk_b\,(t^b\,dt)$ and, for instance, $\zl=\sum_{i\in\N}{\frak K}^i\za_i$. Then $$g(f,\zw)(v,\zl):=(\sum_{i\in\N}{\frak K}^i\za_i)(\zw)f(v):=(\sum_{a=0}^N{\frak K}^ak_a)f(v)\in W\;$$ defines a map between the mentioned spaces. Eventually, a degree $-1$ element of $L\0\zW_1^\star$ (resp., an {\small MC} element of $L\0\zW_1^\star$, a Quillen homotopy) is sent to an element of $\h_{\op{GC}}(\op{Zin}^c({\cal V})\0\zL,\op{Zin}^c{\cal W})$ (resp., an element of $\h_{\op{DGC}}(\op{Zin}^c({\cal V})\0\zL,\op{Zin}^c{\cal W})$, a concordance).\medskip

The relationship between Quillen and gauge homotopy is (at least on the chosen level of rigor) much clearer. Indeed, an element $\bar\zg\in\op{MC}_1(\bar L) = \op{MC}(L\0\zW_1^\star)$ can be decomposed as
$$ \bar\zg=\zg(t)\otimes 1+ r(t)\otimes dt,$$
where $t\in[0,1]$ is the coordinate of $\Delta^1$. When unraveling the {\small MC} equation of the $\bar{\cal L}_i$ according to the powers of $dt$, one gets
\begin{equation}\label{Getzler-Shoiket_homotopy}
\begin{array}{l}
\sum_{p=1}^{\infty}\frac{1}{p!}{\mathcal{L}}_p(\zg(t),...,\zg(t))=0\;,\\
\\
\frac{d\zg}{dt}=-\sum^{\infty}_{p=0}\frac{1}{p!}\mathcal{L}_{p+1}(\zg(t),...,\zg(t),r(t))\;.
\end{array}
\end{equation}
A (nonobvious) direct computation allows to see that the latter {\small ODE}, see Definition \ref{DolSho} of gauge homotopies and Equations (\ref{FlowEq}) and (\ref{SpecVectfield}), is dual (up to dimensional issues) to the {\small ODE} (\ref{DECon}), see Proposition \ref{CharConcord} that characterizes concordances.

\section{Infinity category of Leibniz infinity algebras}

We already observed that vertical composition of concordances is not well-defined and that Leibniz infinity algebras should form an infinity category. It is instructive to first briefly look at infinity homotopies between infinity morphisms of {\small DG} algebras.

\subsection{{\small DG} case}

Remember that infinity homotopies can be viewed as integral curves of specific vector fields $V_r$ of the {\small MC} quadric (with obvious endpoints). In the {\small DG} case, we have, for any $r\in L_0$, $$V_r:L_{-1}\ni\za\mapsto V_r(\za) = -{\cal L}_1(r)-{\cal L}_2(\za,r)\in L_{-1}\;.$$ In view of the Campbell-Baker-Hausdorff formula,

$$ \exp(t V_r)\circ \exp(t V_s) = \exp (t V_r + t V_s + 1/2\; t^2 [V_r, V_s]+ ...)\;. $$
The point is that $$V:L_0\to \op{Vect}(L_{-1})$$ is a Lie algebra morphism -- also after restriction to the {\small MC} quadric; we will not detail this nonobvious fact. It follows that

$$\exp(t V_r)\circ \exp(t V_s) = \exp (t V_{r+s + 1/2\; t [r,s]+ ...})\;.$$
If we accept, as mentioned previously, time-dependent $r$-s, the problem of the vertical composition of homotopies is solved in the considered {\small DG} situation: the integral curve of the composed homotopy of two homotopies $\exp(t V_s)$ (resp., $\exp(t V_r)$) between morphisms $f,g$ (resp., $g,h$) is given by

$$c(t)=(\exp(t V_r)\circ \exp(t V_s))(f) = \exp (t V_{r+s + 1/2\; t [r,s]+ ...})(f)\;.$$

Note that this vertical composition is not associative. Moreover, the preceding approach does not go through in the infinity situation (note e.g. that in this case $L_0$ is no longer a Lie algebra). This again points out that homotopy algebras form infinity categories.

\subsection{Shortcut to infinity categories}\label{InftyCatIntro}

This subsection is a short digression that should allow us to grasp the spirit of infinity categories. For additional information, we refer the reader to \cite{Gro10}, \cite{Fin11} and \cite{Nog12}.\medskip


{\it Strict $n$-categories} or {\it strict $\zw$-categories} (in the sense of strict infinity categories) are well understood, see e.g. \cite{KMP11}. Roughly, they are made up by 0-morphisms (objects), 1-morphisms (morphisms between objects), 2-morphisms (homotopies between morphisms)..., up to $n$-morphisms, except in the $\zw$-case, where this upper bound does not exist. All these morphisms can be composed in various ways, the compositions being associative, admitting identities, etc. However, in most cases of higher categories these defining relations do not hold strictly. A number of concepts of weak infinity category, e.g. infinity categories in which the structural relations hold up to coherent higher homotopy, are developed in literature. Moreover, an $(\infty,r)$-category is an infinity category, with the additional requirement that all $j$-morphisms, $j>r$, be invertible. In this subsection, we actually confine ourselves to {\it $(\infty,1)$-categories, which we simply call $\infty$-categories}.

\subsubsection{First examples}

\begin{ex} {\it $\infty$-categories should include ordinary categories}.\end{ex}

There is another natural example of infinity category. When considering all the paths in $T\in{\tt Top}$, up to homotopy (for fixed initial and final points), we obtain the {\it fundamental groupoid} $\zP_1(T)$ of $T$. Remember that the usual `half-time' composition of paths is not associative, whereas the induced composition of homotopy classes is. Hence, $\zP_1(T)$, with the points of $T$ as objects and the homotopy classes of paths as morphisms, is (really) a category in which all morphisms are invertible. To encode more information about $T$, we can use a 2-category, the {fundamental 2-groupoid} $\zP_2(T)$, whose 0-morphisms (resp., 1-morphisms, 2-morphisms) are the points of $T$ (resp., the paths between points, the homotopy classes of homotopies between paths), the composition of 1-morphisms being associative only up to a 2-isomorphism. More generally, we define the {\it fundamental $k$-groupoid} $\zP_k(T)$, in which associativity of order $j\le k-1$ holds up to a $(j+1)$-isomorphism. Of course, if we increase $k$, we grasp more and more information about $T$. The {\it homotopy principle} says that the weak fundamental infinity groupoid $\zP_{\infty}(T)$ recognizes $T$, or, more precisely, that {\it $(\infty,0)$-categories are the same as topological spaces}. Hence,

\begin{ex} $\infty$-categories, i.e. $(\infty,1)$-categories, should contain topological spaces.\end{ex}

\subsubsection{Kan complexes, quasi-categories, nerves of groupoids and of categories}

Let us recall that the {\it nerve functor} $N:{\tt Cat}\rightarrow {\tt SSet}$, provides a fully faithful embedding of the category ${\tt Cat}$ of all (small) categories into ${\tt SSet}$ and remembers not only the objects and morphisms, but also the compositions. It associates to any ${\tt C}\in{\tt Cat}$ the simplicial set $$(N{\tt C})_n=\{C_0\to C_1\to\ldots\to C_n\}\;,$$ where the sequence in the {\small RHS} is a sequence of composable {\tt C}-morphisms between objects $C_i\in{\tt C}$; the face (resp., the degeneracy) maps are the compositions and insertions of identities. Let us also recall that the {\it $r$-horn} $\zL^r[n]$, $0\le r\le n$, of $\zD[n]$ is `the part of the boundary of $\zD[n]$ that is obtained by suppressing the interior of the $(n-1)$-face opposite to $r$'. More precisely, the $r$-horn $\zL^r[n]$ is the simplicial set, whose nondegenerate $k$-simplices are the injective order-respecting maps $[k]\to [n]$, except the identity and the map $\zd^r:[n-1]\rightarrowtail [n]$ whose image does not contain $r$.\medskip

We now detail four different situations based on the properties `Any (inner) horn admits a (unique) filler'.

\begin{defi} A simplicial set $S\in{\tt SSet}$ is \emph{fibrant} and called a \emph{Kan complex}, if the map $S\to \star$, where $\star$ denotes the terminal object, is a Kan fibration, i.e. has the right lifting property with respect to all canonical inclusions $\zL^r[n]\subset \zD[n]$, $0\le r\le n$, $n>0$. In other words, $S$ is a Kan complex, if {any horn} $\zL^r[n]\to S$ can be extended to an $n$-simplex $\zD[n]\to S$, i.e. if {\sf any horn in $S$ admits a filler}.\end{defi}

The following result is well-known and explains that a simplicial set is a nerve under a quite similar extension condition.

\begin{prop} A simplicial set $S$ is the nerve $S\simeq N{\tt C}$ of some category ${\tt C}$, if and only if {\sf any inner horn $\zL^r[n]\to S$, $0<r<n$, has a unique filler $\zD[n]\to S$}. \end{prop}

Indeed, it is quite obvious that for $S=N{\tt C}\in{\tt SSet}$, an inner horn $\zL^1[2]\to N{\tt C}$, i.e. two {\tt C}-morphisms $f: C_0\to C_1$ and $g:C_1\to C_2$, has a unique filler $\zD[2]\to N{\tt C}$, given by the edge $h=g\circ f:C_0\to C_2$ and the `homotopy' $\id:h\Rightarrow g\circ f$ (1).\smallskip

As for Kan complexes $S\in{\tt SSet}$, the filler property for an outer horn $\zL^0[2]\to S$ (resp., $\zL^2[2]\to S$) implies for instance that a horn $f:s_0\to s_1$, $\id:s_0\to s_2=s_0$ (resp., $\id: s'_0\to s'_2=s_0'$, $g: s'_1\to s'_2$) has a filler, so that any map has a `left (resp., right) inverse' (2).\medskip

It is clear that simplicial sets $S_0,S_1,S_2,\,\ldots$ are candidates for $\infty$-categories. In view of the last remark (2), Kan complexes model $\infty$-groupoids. Hence, fillers for outer horns should be omitted in the definition of $\infty$-categories. On the other hand, $\infty$-categories do contain homotopies $\zh: h\Rightarrow g\circ f$, so that, due to (1), uniqueness of fillers is to be omitted as well. Hence, the

\begin{defi} A simplicial set $S\in{\tt SSet}$ is an \emph{$\infty$-category} if and only if {\sf any inner horn $\zL^r[n]\to S$, $0<r<n$, admits a filler $\zD[n]\to S$}. \end{defi}

We now also understand the

\begin{prop} A simplicial set $S$ is the nerve $S\simeq N{\tt G}$ of some groupoid ${\tt G}$, if and only if {\sf any horn $\zL^r[n]\to S$, $0\le r\le n$, $n>0$, has a unique filler $\zD[n]\to S$}.\end{prop}

Of course, Kan complexes, i.e. $\infty$-groupoids, $(\infty,0)$-categories, or, still, topological spaces, are $\infty$-categories. Moreover, nerves of categories are $\infty$-categories. Hence, the requirement that topological spaces and ordinary categories should be $\infty$-categories are satisfied. Note further that what we just defined is a model for $\infty$-categories called {\it quasi-categories} or {\it weak Kan complexes}.

\subsubsection{Link with the intuitive picture of an infinity category}\label{InftyCatComp}

In the following, we explain that the preceding model of an $\infty$-category actually corresponds to the intuitive picture of an $(\infty,1)$-category, i.e. that in an $\infty$-category all types of morphisms do exist, that all $j$-morphisms, $j>1$, are invertible, and that composition of morphisms is defined and is associative only up to homotopy. This will be illustrated by showing that any $\infty$-category has a homotopy category, which is an ordinary category.\medskip

We denote simplicial sets by $S, S',\ldots$, categories by ${\tt C},{\tt D},\dots$, and $\infty$-categories by ${\tt S},{\tt S'},\ldots$\medskip

Let {\tt S} be an $\infty$-category. Its {\it 0-morphisms} are the elements of ${\tt S}_0$ and its {\it 1-morphisms} are the elements of ${\tt S}_1$. The {\it source} and {\it target} maps $\zs,\zt$ are defined, for any 1-morphism $f\in {\tt S}_1$, by $\zs f=d_1f\in{\tt S}_0$, $\zt f=d_0f\in{\tt S}_0$, and the {\it identity map} is defined, for any 0-morphism $s\in {\tt S}_0$, by $\id_s=s_0s\in {\tt S}_1$, with self-explaining notation. In the following, we denote a 1-morphism $f$ with source $s$ and target $s'$ by $f:s\to s'$. In view of the simplicial relations, we have $\zs \id_s=d_1s_0s=s$ and $\zt \id_s=d_0s_0s=s$, so that $\id_s:s\to s$.\medskip

Consider now two morphisms $f:s\to s'$ and $g:s'\to s''$. They define an inner horn $\zL^1[2]\to {\tt S}$, which, as {\tt S} is an $\infty$-category, admits a filler $\zf:\zD[2]\to {\tt S}$, or $\zf\in {\tt S}_2$. The face $d_1\zf\in{\tt S}_1$ is of course a candidate for the composite $g\circ f.$

\begin{rem}\label{CompInftyCat} Since the face $h:=d_1\zf$ of any filler $\zf$ is a (candidate for the) composite $g\circ f$, composites of morphisms are in $\infty$-categories not uniquely defined. We will show that they are determined only up to `homotopy'.\end{rem}

\begin{defi}\label{2MorphInftyCat} Let {\tt S} be an $\infty$-category and let $f,g:s\to s'$ be two morphisms. A \emph{2-morphism} or \emph{homotopy} $\zf:f\Rightarrow g$ between $f$ and $g$ is an element $\zf\in{\tt S}_2$ such that $d_0\zf=g, d_1\zf=f, d_2\zf=\id_{s}$.\end{defi}

Indeed, if there exists such a 2-simplex $\zf$, there are two candidates for the composite $g\circ \id_{s}$, namely $f$ and, of course, $g$. If we wish now that all the candidates be homotopic, the existence of $\zf$ must entail that $f$ and $g$ are homotopic -- which is the case in view of Definition \ref{2MorphInftyCat}. If $f$ is homotopic to $g$, we write $f\simeq g$.

\begin{prop} The homotopy relation $\simeq$ is an equivalence in ${\tt S}_1$.\end{prop}

\begin{proof} Let $f:s\to s'$ be a morphism and consider $\id_f:=s_0f\in{\tt S}_2$. It follows from the simplicial relations that $d_0\id_f=f,d_1\id_f=f,d_2\id_f=s_0s=\id_s$, so that $\id_f$ is a homotopy between $f$ and $f$. To prove that $\simeq$ is symmetric, let $f,g:s\to s'$ and assume that $\zf$ is a homotopy from $f$ to $g$. We then have an inner horn ${\psi}:\zL^2[3]\to {\tt S}$ such that $d_0\psi=\zf$, $d_1\psi=\id_g$, and $d_3\psi=\id_{\id_s}=:\id^2_s$. The face $d_2\Psi$ of a filler $\Psi:\zD[3]\to {\tt S}$ is a homotopy from $g$ to $f$. Transitivity can be obtained similarly. \end{proof}

\begin{defi} The \emph{homotopy category} ${\tt Ho}(\tt S)$ of an $\infty$-category $\,{\tt S}$ is the (ordinary) category with objects the objects $s\in {\tt S}_0$, with morphisms the homotopy classes $[f]$ of morphisms $f\in {\tt S}_1$, with composition $[g]\circ [f]=[g\circ f]$, where $g\circ f$ is any candidate for the composite in ${\tt S}$, and with identities $\op{Id}_s=[\id_s]$.\end{defi}

To check that this definition makes sense, we must in particular show that all composites $g\circ f$, see Remark \ref{CompInftyCat}, are homotopic. Let thus $\zf_1,\zf_2\in{\tt S}_2$ be two 2-simplices such that $(d_0\zf_1,d_1\zf_1,d_2\zf_1)=(g,h_1,f)$ and $(d_0\zf_2,d_1\zf_2,d_2\zf_2)=(g,h_2,f)$, so that $h_1$ and $h_2$ are two candidates. Consider now for instance the inner horn $\psi:\zL^2[3]\to {\tt S}$ given by $\psi=(\zf_1,\zf_2,\bullet,\id_f)$. The face $d_2\Psi$ of a filler $\Psi:\zD[3]\to {\tt S}$ is then a homotopy from $h_2$ to $h_1$. To prove that the composition of morphisms in ${\tt Ho}({\tt S})$ is associative, one shows that candidates for $h\circ (g\circ f)$ and for $(h\circ g)\circ f$ are homotopic (we will prove neither this fact, nor the additional requirements for ${\tt Ho}({\tt S})$ to be a category).

\begin{rem} It follows that in an $\infty$-category composition of morphisms is defined and associative only up to homotopy.\end{rem}

We now comment on higher morphisms in $\infty$-categories, on their composites, as well as on invertibility of $j$-morphisms, $j>1$.

\begin{defi} Let $\zf_1:f\Rightarrow g$ and $\zf_2:f\Rightarrow g$ be 2-morphisms between morphisms $f,g:s\to s'$. A \emph{3-morphism} $\Phi:\zf_1\Rrightarrow \zf_2$ is an element $\Phi\in{\tt S}_3$ such that $d_0\Phi= \id_g$, $d_1\Phi =\zf_2$, $d_2\Phi =\zf_1$, and $d_3\Phi= \id^2_s$.\end{defi} Roughly, a 3-morphism is a 3-simplex with faces given by sources and targets, as well as by identities. Higher morphisms are defined similarly \cite{Gro10}.\medskip

As concerns composition and invertibility, let us come back to transitivity of the homotopy relation. There we are given 2-morphisms $\zf_1:f\Rightarrow g$ and $\zf_2:g\Rightarrow h$, and must consider the inner horn $\psi=(\zf_2,\bullet,\zf_1,\id^2_s)$. The face $d_1\Psi$ of a filler $\Psi$ is a homotopy between $f$ and $h$ and is a candidate for the composite $\zf_2\circ\zf_1$ of the 2-morphisms $\zf_1,\zf_2$. If we now look again at the proof of symmetry of the homotopy relation and denote the homotopy from $g$ to $f$ by $\psi'$, we see that $\psi\circ\psi'\simeq\id_g$. We obtain similarly that $\psi'\circ\psi\simeq\id_f$, so that 2-morphisms are `invertible'.

\begin{rem} Eventually, all the requirements of the intuitive picture of an $\infty$-category are encoded in the existence of fillers of inner horns.\end{rem}

\subsection{Infinity groupoid of infinity morphisms between Leibniz infinity algebras}\label{sec_VcoH}

\subsubsection{Quasi-category of homotopy Leibniz algebras}\label{InftyCatInftyAlg}

Let $\zW^\star_\bullet$ be the {\small SDGCA} introduced in Subsection \ref{SimpicialDR}. The `Yoneda embedding' of $\zW^\star_\bullet$ viewed as object of ${\tt SSet}$ and ${\tt DGCA}$, respectively, gives rise to an adjunction that is well-known in Rational Homotopy Theory: $$\zW^\star:{\tt SSet}\rightleftarrows {\tt DGCA}^{\op{op}}:\op{Spec}_{\bullet}\;.$$ The functor $\zW^\star=\h_{\tt SSet}(-,\zW^\star_\bullet)=:\op{SSet}(-,\zW^\star_\bullet)$ associates to any $S_\bullet\in{\tt SSet}$ its Sullivan {\small DGCA} $\zW^\star(S_\bullet)$ of piecewise polynomial differential forms, whereas the functor $\op{Spec}_{\bullet}=\h_{\tt DGCA}(-,\zW^\star_\bullet)$ assigns to any $A\in {\tt DGCA}$ its simplicial spectrum $\op{Spec}_{\bullet}(A)$.\medskip

Remember now that an $\infty$-homotopy between $\infty$-morphisms between two Leibniz infinity algebras $V,W$, is an element in $\op{MC}_1({\bar L})=\op{MC}(L\0\zW^\star_1)$, where $L=L(V,W)$.\medskip

The latter set is well-known from integration of $L_{\infty}$-algebras. Indeed, when looking for an integrating topological space or simplicial set of a positively graded $L_{\infty}$-algebra $L$ of finite type (degree-wise of finite dimension), it is natural to consider the simplicial spectrum of the corresponding quasi-free {\small DGCA} $\op{Com}(s^{-1}L^*)$. The dual of Equation (\ref{NatBijCoAlg}) yields $$\op{Spec}_\bullet(\op{Com}(s^{-1}L^*))=\h_{\tt DGCA}(\op{Com}(s^{-1}L^*),\zW^\star_\bullet)\simeq \op{MC}(L\0 \zW^\star_\bullet)\;.$$ The integrating simplicial set of a nilpotent $L_{\infty}$-algebra $L$ is actually homotopy equivalent to $\op{MC}_\bullet({\bar L}):=\op{MC}(L\0\zW^\star_\bullet)$ \cite{Get09}. It is clear that the structure maps of $\op{MC}_\bullet({\bar L})\subset L\0\zW^\star_\bullet$ are $\tilde{d}^n_i=\id\0 d^n_i$ and $\tilde{s}^n_i=\id\0 s^n_i$, where $d^n_i$ and $s^n_i$ were described in Subsection \ref{SimpicialDR}.\medskip

Higher homotopies ($n$-homotopies) are usually defined along the same lines as standard homotopies (1-homotopies), i.e., e.g., as arrows depending on parameters in $I^{\times n}$ (or $\zD^n$) instead of $I$ (or $\zD^1$) \cite{Lei03}. Hence,

\begin{defi} \emph{$\infty$-$n$-homotopies} $($\emph{$\infty$-$(n+1)$-morphisms}$)$ between given Leibniz infinity algebras $V,W$ are Maurer-Cartan elements in $\op{MC}_n(\bar L)=\op{MC}(L\0\zW^\star_n)$, where $L=L(V,W)$ and $n\ge 0$.\label{InftyNHomot}\end{defi}

Indeed, $\infty$-1-morphisms are just elements of $\op{MC}(L)$, i.e. standard $\infty$-morphisms between $V$ and $W$.\medskip

Note that if $\tt S$ is an $\infty$-category, the set of $n$-morphisms, with varying $n\ge 1$, between two fixed objects $s,s'\in{\tt S}_0$ can be shown to be a Kan complex \cite{Gro10}. The simplicial set $\op{MC}_\bullet(\bar L)$, whose $(n-1)$-simplices are the $\infty$-$n$-morphisms between the considered Leibniz infinity algebras $V,W$, $n\ge 1$, is known to be a Kan complex ($(\infty,0)$-category) as well \cite{Get09}.

\begin{rem} We interpret this result by saying that Leibniz infinity algebras and their infinity higher morphisms form an $\infty$-category $($$(\infty,1)$-category$)$. Further, as mentioned above and detailed below, composition of homotopies is encrypted in the Kan property.\end{rem}

Note that $\op{MC}_\bullet(\bar L)$ actually corresponds to the `d\'ecalage', the `down-shifting', of the simplicial set $\tt S$.\smallskip

Let us also \emph{emphasize} that Getzler's results are valid only for nilpotent $L_{\infty}$-algebras, hence in principle not for $L$, which is only complete (an $L_{\infty}$-algebra is \emph{pronilpotent}, if it is complete with respect to its lower central series, i.e. the intersection of all its compatible filtrations, and it is \emph{nilpotent}, if its lower central series eventually vanishes). However, for our remaining concern, namely the explanation of homotopies and their compositions in the 2-term Leibniz infinity algebra case, this difficulty is irrelevant. Indeed, when interpreting the involved series as formal ones and applying the thus obtained results to the 2-term case, where series become finite for degree reasons, we recover the results on homotopies and their compositions conjectured in \cite{BC04}. An entirely rigorous approach to these issues is being examined in a separate paper: it is rather technical and requires applying Henriques' method or working over an arbitrary local Artinian algebra.






\subsubsection{Kan property}\label{Kan property}

Considering our next purposes, we now review and specify the proof of the Kan property of $\op{MC}_\bullet(\bar L)$ \cite{Get09}. As announced above, to facilitate comparison, we adopt the conventions of the latter paper and apply the results mutatis mutandis and formally to our situation. In particular, we work in \ref{Kan property} with the cohomological version of Lie infinity algebras ($k$-ary bracket of degree $2-k$), together with Getzler's sign convention for the higher Jacobi conditions, see \ref{Gauge}, and assume that $\K=\R$.\medskip

Let us first recall that the lower central filtration of $(L,{\cal L}_i)$ is given by $F^1L=L$ and
$$F^iL=\sum_{i_1+\cdots+i_k=i}{\cal L}_k(F^{i_1}L,\cdots,F^{i_k}L),~~i>1\;.$$ In particular, $F^2L={\cal L}_2(L,L)$, $F^3L={\cal L}_2(L,{\cal L}_2(L,L))+{\cal L}_3(L,L,L)$, ..., so that $F^kL$ is spanned by all the nested brackets containing $k$ elements of $L$. Due to nilpotency, $F^iL=\{0\}$, for $i\gg$.\medskip

To simplify notation, let $\zd$ be the differential ${\cal L}_1$ of $L$, let $d$ be the de Rham differential of $\zW^\star_n=\zW^\star(\zD^n)$, and let $\bar\zd+\bar d$ be the differential $\bar{\cal L}_1=\delta\0 \id+\id\0 d$ of $L\0\zW^\star(\zD^n)$. Set now, for any $n\ge 0$ and any $0\le i\le n$,
$$\op{mc}_n(\bar L):=\{(\bar\zd+\bar d)\beta:\beta \in (L\otimes\Omega^{\star}(\zD^n))^0\}\;\text{and}\; \op{mc}_n^i(\bar L):=\{(\bar\zd+\bar d)\beta:\beta \in (L\otimes\Omega^{\star}(\zD^n))^0, \;\bar\ze^i_n\zb=0\}\;,$$ where $\bar\ze_n^i:=\id\0 \ze_n^i$ is the canonical extension of the evaluation map $\ze_n^i:\zW^\star(\zD^n)\to \K$, see \ref{SimpicialDR}. \begin{rem} In the following, we use the extension symbol `bar' only when needed for clarity.\end{rem}

$\bullet\quad$  There exist fundamental bijections \bea B^i_n: \op{MC}_n(\bar L)\stackrel{\sim}{\longrightarrow} \op{MC}(L)\times \op{mc}_n^i(\bar L)\subset \op{MC}(L)\times\op{mc}_n(\bar L)\label{important_bijection}\;.\eea

The proof uses the operators $$h_n^i:~\Omega^{\star}(\Delta^n)\to \zW^{\star-1}(\Delta^n)$$ defined as follows. Let $\vec t=[t_0,\ldots, t_n]$ be the coordinates of $\Delta^n$ (with $\sum_i t_i=1$) and consider the maps $\phi_n^i:~I\times\Delta^n\ni (u,\vec t)\mapsto u\vec t+(1-u)\vec e_i\in \Delta^n$. They allow to pull back a polynomial differential form on $\Delta^n$ to a polynomial differential form on $I\times \Delta^n$. The operators $h_n^i$ are now given by \bea h_n^i\go=\int_I~(\phi_n^i)^*\go\;.\nn\eea They satisfy the relations
\bea \{d,h_n^i\}=\id_n-\ep_n^i,~~~~\{h_n^i,h_n^j\}=0,~~~~~\ep_n^ih_n^i=0\;,\label{ibp_proper}\eea where $\{-,-\}$ is the graded commutator (remember that $\ze_n^i$ vanishes in nonzero cohomological degree). The first relation is a higher dimensional analogue of $$\{d,\int_0^t\}\go=\{d,\int_0^t\}(f(u)+g(u)du)=d\int_0^tg(u)du+\int_0^td_uf\,du=g(t)dt+f(t)-f(0)=\go-\ep_1^0\go\;,$$ where $\zw\in\zW^\star(I).$\medskip

The natural extensions of $d,$ $h_n^i,$ and $\ze_n^i$ to $L\0\zW^\star(\zD^n)$ satisfy the same relations, and, since we obviously have $\zd h_n^i=-\,{h}_n ^i\zd$, the first relation holds in the extended setting also for $d$ replaced by $\zd+d$.\medskip

Define now $B_n^i$ by
\be B_n^i:\op{MC}_n(\bar L)\ni\za\mapsto B_n^i\za:=(\ep_n^i\za,(\gd+d)h_n^i\za)\in \op{MC}(L)\times \op{mc}_n^i(\bar L)\;.\label{BDir}\ee Observe that $\za\in (L\0\zW^\star(\zD^n))^1$ reads $\za=\sum_{k=0}^n\za^k$, $\za^k\in L_{1-k}\0\zW^k(\zD^n)$, so that $\ze_n^i\za=\ze_n^i\za^0\in L_1$. Moreover, it follows from the definition of the extended $L_{\infty}$-maps $\bar{\cal L}_i$ that \be\label{EpsilonMC}\sum_{i\ge 1}\frac{1}{i!}{\cal L}_i(\ze_n^i\za,\ldots,\ze_n^i\za)=\ze_n^i\sum_{i\ge 1}\frac{1}{i!}\bar{\cal L}_i(\za,\ldots,\za)=0\;.\ee In view of the last equation (\ref{ibp_proper}), the second component of $B_n^i\za$ is clearly an element of $\textrm{mc}_n^i(\bar L)$.\medskip

The construction of the inverse map is based upon a method similar to the iterative approximation procedure that allows us to prove the fundamental theorem of {\small ODE}-s. More precisely, consider the Cauchy problem $y\,'(t)=F(t,y(t))$, $y(0)=Y$, i.e. the integral equation $$y(s)=Y+\int_0^sF(t,y(t))dt\;.$$ Choose now the `Ansatz' $y_0(s)=Y$ and define inductively $$y_k(s)=Y+\int_0^sF(t,y_{k-1}(t))dt\;,$$ $k\ge 1$. It is well-known that the $y_k$ converge to a function $y$, which is the unique solution and depends continuously on the initial value $Y$.\medskip

Note now that, if we are given $\mu\in\textrm{MC}(L)$ and $\nu=(\zd+d)\zb\in \textrm{mc}^i_n(\bar L)$, a solution $\za\in\op{MC}_n(\bar L)$ -- i.e. an element $\za\in(L\0\zW^\star(\zD^n))^1$ that satisfies $$(\zd+d)\za+\sum_{i\ge 2}\frac{1}{i!}\bar{\cal L}_i(\za,\ldots,\za)=:(\zd+d)\za+\bar{R}(\za)=0\;\;\;\text{--}\;$$ such that $\ze_n^i\za=\zm$ and $(\zd+d)h_n^i\za=\zn$, also satisfies the integral equation \be\za=\id_n\za=\{\zd+d,h_n^i\}\za+\ze_n^i\za=\zm+\zn+h_n^i(\zd+d)\za=\zm+\zn-h_n^i\bar R(\za)\;.\label{IntEq}\ee We thus choose the `Ansatz' $\ga_0=\mu+\nu$ and set $\ga_{k}=\ga_0-h_n^i\bar R(\za_{k-1}),$ $k\ge 1$. It is easily seen that, in view of nilpotency, this iteration stabilizes, i.e. $\za_{k-1}=\za_k=\ldots =:\ga$, for $k\gg$, or, still, \bea \ga=\ga_0-h_n^i\bar R(\za)\;.\label{stable}\eea The limit $\za$ is actually a solution in $\op{MC}_n(\bar L)$. Indeed, remember first that the generalized curvature
\bea \bar{\cal F}(\ga)=(\gd+d)\ga+\bar{R}(\za)=(\gd+d)\ga+\sum_{i\ge 2}\frac{1}{i!}\bar{\cal L}_i(\ga,\ldots,\ga)\;,\nn\eea
whose zeros are the {\small MC} elements, satisfies, just like the standard curvature, the Bianchi identity
\bea (\gd+d)\bar{\cal F}(\za)+\sum_{k\ge 1}\frac{1}{k!}\bar{\cal L}_{k+1}(\ga,\ldots,\ga,\bar{\cal F}(\za))=0\;.\label{biancchi}\eea
It follows from (\ref{stable}) and (\ref{ibp_proper}) that
\bea \bar{\cal F}(\ga)=(\gd+d)(\ga_0-h_n^i\bar R(\ga))+\bar{R}(\ga)=(\gd+d)\mu+h_n^i(\gd+d)\bar R(\za)+\ep_n^i\bar R(\za)\;.\nn\eea From Equation (\ref{EpsilonMC}) we know that $\ze_n^i\bar R(\za)=R(\ze_n^i\za)=R(\zm)$, with self-explaining notation. As for $\ze_n^i\za=\zm$, note that $\ze_n^i\zm=\zm$ and that $$\ze_n^i\zn=\ze_n^i(\zd+d)\zb=\ze_n^i(\zd+d)\sum_{k=0}^n\zb^k,\; \zb^k\in L_{-k}\0\zW^k(\zD^n),\;\text{ so that }\; \ze_n^i\zn=\ze_n^i\zd\zb^0=\zd\ze_n^i\zb^0=\zd\ze_n^i\zb=0\;.$$ Hence, $$\bar{\cal F}(\ga)=(\gd+d)\mu+h_n^i(\gd+d)(\bar{\cal F}(\ga)-(\gd+d)\ga)+R(\mu)$$ $$={\cal F}(\mu)+h_n^i(\gd+d)\bar{\cal F}(\ga)=-h_n^i\sum_{k\ge 1}\frac{1}{k!}\bar{\cal L}_{k+1}(\ga,\ldots,\ga,\bar{\cal F}(\za))\;,$$ in view of (\ref{biancchi}). Therefore, $\bar{\cal F}(\ga)\in F^i\bar L$, for arbitrarily large $i$, and thus $\za\in\op{MC}_n(\bar L)$. This completes the construction of maps \be\label{BInv}{\cal B}_n^i:\op{MC}(L)\times \op{mc}_n^i(\bar L)\to \op{MC}_n(\bar L)\;.\ee

We already observed that $\ze_n^i{\cal B}_n^i(\zm,\zn)=\ze_n^i\za=\zm$. In fact, $B_n^i{\cal B}_n^i=\id$, so that $B_n^i$ is surjective. Indeed, Equations (\ref{stable}) and (\ref{ibp_proper}) imply that $$(\zd+d)h_n^i\za=-h_n^i(\zd+d)\za_0+\za_0-\ze_n^i\za_0=-h_n^i\zd\zm+\zn=\zn\;.$$ As for injectivity, if $B_n^i\za=B_n^i\za'=:(\zm,\zn)$, then both, $\ga$ and $\ga'$, satisfy Equation (\ref{IntEq}). It is now quite easily seen that nilpotency entails that $\ga=\ga'$.\bigskip

$\bullet\quad$  The bijections $$B_n^i:\op{MC}_n(\bar L)\to \op{MC}(L)\times \op{mc}_n^i(\bar L)$$ allow us to prove the Kan property for $\op{MC}_\bullet(\bar L)$. The extension of a horn in $\op{SSet}(\zL^i[n],\op{MC}_\bullet(\bar L))$ will be performed as sketched by the following diagram:
\bea
\xymatrix{
\textrm{SSet}(\Lambda^i[n],\textrm{MC}_{\bullet}(\bar L)) \ar@{.>}^-{ }[r]\ar@{->}_-{ }[d]&\textrm{MC}_n(\bar L)\\
\textrm{SSet}(\Lambda^i[n],\textrm{MC}(L)\times\textrm{mc}_{\bullet}(\bar L))\ar@{->}^-{}[r]&\textrm{MC}(L)\times\textrm{mc}^i_n(\bar L)\ar@{->}_-{}[u]
}\label{commute_square}\eea

Of course, the right arrow is nothing but ${\cal B}_n^i$.\medskip

$\star\quad$  As for the left arrow, imagine, for simplicity, that $i=1$ and $n=2$, and let $$\za\in \op{SSet}(\zL^1[2], \op{MC}_\bullet(\bar L))\;.$$ The restrictions $\za|_{01}$ and $\za|_{12}$ to the 1-faces $01$, $12$ (compositions of the natural injections with $\za$) are elements in $\op{MC}_1(\bar L)$, so that the map $B_1^1$ sends $\za|_{01}$ to $(\zm,\zn)$ in $\op{MC}(L)\times\op{mc}_1(\bar L)$ (and similarly $B_1^0(\za|_{12})=(\zm',\zn')\in \op{MC}(L)\times\op{mc}_1(\bar L)$). Of course, $\zm=\ze_1^1(\za|_{01})=\ze_1^0(\za|_{12})=\zm'$. Since $\zn=(\zd+d)\zb$ and $\zb(1)=\ze_1^1\zb=0$, we find $\zn(1)=\ze_1^1\zn=0$ (and similarly $\zn'=(\zd+d)\zb'$ and $\zb'(1)=\zn'(1)=0$). Thus, $$(\zm;\zn,\zn')\in\op{SSet}(\zL^1[2],\op{MC}(L)\times\op{mc}_\bullet(\bar L))\;,$$ which explains the left arrow.\medskip

$\star\quad$ For the bottom arrow, let again $i=1,n=2$. Since $\zm$ is constant, it can be extended to the whole simplex. To extend $(\zn,\zn')$, it actually suffices to extend $(\zb,\zb')$. Indeed, restriction obviously commutes with $\zd$. As for commutation with $d$, remember that $\zW^\star:\zD^{\op{op}}\to {\tt DGCA}$ and that the {\tt DGCA}-map $d^2_2=\zW^\star(\zd^2_2)$ sets the component $t_2$ to 0. Hence, $d^2_2$ coincides with restriction to $01$ and commutes with $d$. Let now $\bar\zb$ be an extension of $(\zb,\zb')$. Since $$(d\bar\zb)|_{01}=d^2_2d\bar\zb=dd^2_2\bar\zb=d\zb$$ and similarly $(d\bar\zb)|_{12}=d\zb'$.\medskip

It now remains to explain that an extension $\bar\zb$ does always exist. Consider the slightly more general extension problem of three polynomial differential forms $\zb_0$, $\zb_1,$ and $\zb_2$ defined on the 1-faces $12$, $02$, and $01$ of the 2-simplex $\zD^2$, respectively (it is assumed that they coincide at the vertices). Let $\zp_2:\zD^2\to 01$ be the projection defined, for any $\vec t=[t_0,t_1,t_2]$, as the intersection of the line $u\vec t+(1-u)\vec e_2$ with $01$. This projection is of course ill-defined at $\vec t=\vec e_2$. In coordinates, we get $$\zp_2:[t_0,t_1,t_2]\mapsto [t_0/(1-t_2), t_1/(1-t_2)]\;.$$ It follows that the pullback $\zp_2^*\zb_2$ is a rational differential form with denominator $(1-t_2)^N$, for some integer $N.$ Hence, $$\zg_2:=(1-t_2)^N\zp_2^*\zb_2$$ is a polynomial differential form on $\zD^2$ that coincides with $\zb_2$ on $01$. It now suffices to solve the same extension problem as before, but for the forms $\zb_0-\zg_2|_{12}$, $\zb_1-\zg_2|_{02}$, and $0$. When iterating the procedure -- due to Renshaw \cite{Sul77} --, the problem reduces to the extension of $0,0,0$ (since the pullback preserves 0). This completes the description of the bottom arrow, as well as the proof of the Kan property of $\op{MC}_\bullet(\bar L)$.

\section{2-Category of 2-term Leibniz infinity algebras}

Categorification replaces sets (resp., maps, equations) by categories (resp., functors, natural isomorphisms). In particular, rather than considering two maps as equal, one details a way of identifying them. Categorification is thus a sharpened viewpoint that turned out to provide deeper insight. This motivates the interest in e.g. categorified algebras (and in truncated infinity algebras -- see below).\medskip

Categorified Lie algebras were introduced under the name of Lie 2-algebras in [BC04] and further studied in \cite{Roy07}, \cite{SL10}, and \cite{KMP11}. The main result of \cite{BC04} states that Lie 2-algebras and 2-term Lie infinity algebras form equivalent 2-categories. However, {\bf infinity homotopies of 2-term Lie infinity algebras} (resp., {\bf compositions of such homotopies}) {are not explained, but appear as some God-given natural transformations read through this} {\textsc{equivalence}} (resp., compositions are addressed only in \cite{SS07} and performed in the {\textsc{algebraic or coalgebraic settings}}).\medskip

This circumstance is not satisfactory, and {\bf the attempt to improve our understanding of infinity homotopies and their compositions is one of the main concerns of the present paper}. Indeed, in \cite{KMP11} (resp., \cite{BP12}), the authors show that the {\textsc{equivalence}} between $n$-term Lie infinity algebras and Lie $n$-algebras is, for $n> 2$, not as straightforward as expected -- which is essentially due to the largely ignored fact that the category {\tt Vect} $n${\tt -Cat} of linear $n$-categories is symmetric monoidal, but that the corresponding map $\boxtimes:L\times L'\to L\boxtimes L'$ is not an $n$-functor (resp., that the understanding of a concept in the {\textsc{algebraic framework}} is far from implying its comprehension in the infinity context -- a reality that is corroborated e.g. by the comparison of concordances and infinity homotopies).\medskip

In this section, we obtain {\bf explicit formulae for infinity homotopies and their compositions}, applying the \textsc{Kan property} of $\op{MC}_\bullet(\bar L)$ to the 2-term case, thus staying inside the \textsc{infinity setting}.

\subsection{Category of 2-term Leibniz infinity algebras}

For the sake of completeness, we first describe 2-term Leibniz infinity algebras and their morphisms. Propositions \ref{2term_Loday_in_algebra} and \ref{2term_Loday_morphism} are specializations to the 2-term case of Definitions \ref{LeibInftyAlg} and \ref{LeibInftyAlgMorph}; see also \cite{SL10}. The informed reader may skip the present subsection.

\begin{prop}\label{2term_Loday_in_algebra} A \emph{2-term Leibniz infinity algebra} is a graded vector space $V=V_0\oplus V_1$ concentrated in degrees 0 and 1, together with a linear, a bilinear, and a trilinear map $l_1,$ $l_2,$  and $l_3$ on $V$, of degree $|l_1|=-1,$ $|l_2|=0,$ and $|l_3|=1,$ which verify, for any $w,x,y,z\in V_0$ and $h,k\in V_1$,
\begin{enumerate}[(a)]
\item \label{identity_a} $l_1l_2(x,h)=l_2(x,l_1h)\;,$\\
$l_1l_2(h, x) = l_2(l_1h, x)\;,$
\item \label{identity_b}$l_2(l_1h,k)=l_2(h,l_1k)\;,$
\item \label{identity_c}$l_1l_3(x,y,z)=l_2(x,l_2(y,z))-l_2(y,l_2(x,z))-l_2(l_2(x,y),z)\;,$
\item \label{identity_d}$l_3(x,y,l_1h)=l_2(x,l_2(y,h))-l_2(y,l_2(x,h))-l_2(l_2(x,y),h)\;,$\\
$l_3(x, l_1h, y) = l_2(x, l_2(h, y)) - l_2(h, l_2(x, y)) - l_2(l_2(x,h), y)\;,$\\
$l_3(l_1h, x, y) = l_2(h, l_2(x, y)) - l_2(x, l_2(h, y)) - l_2(l_2(h, x), y)\;,$
\item \label{identity_e}$l_2(l_3(w,x,y),z)+l_2(w,l_3(x,y,z))-l_2(x,l_3(w,y,z))+l_2(y,l_3(w,x,z))-l_3(l_2(w,x),y,z)$\\
$+l_3(w,l_2(x,y),z)-l_3(x,l_2(w,y),z)-l_3(w,x,l_2(y,z))+l_3(w,y,l_2(x,z))-l_3(x,y,l_2(w,z))=0\;.$
\end{enumerate}
\end{prop}

\begin{prop}\label{2term_Loday_morphism} An \emph{infinity morphism between 2-term Leibniz infinity algebras} $(V,l_1,l_2,l_3)$ and $(W,m_1,m_2,m_3)$ is made up by a linear and a bilinear map $f_1,f_2$ from $V$ to $W$, of degree $|f_1|=0, |f_2|=1$, which verify, for any $x,y,z\in V_0$ and $h\in V_1$,
\begin{enumerate}[(a)]\label{fdg}
\item\label{morphism_a} $m_1f_1h=f_1l_1h\;,$
\item $m_2(f_1x,f_1y)+m_1f_2(x,y)=f_1l_2(x,y)\;,\label{morphism_b}$
\item \label{morphism_c}$m_2(f_1x,f_1h)=f_1l_2(x,h)-f_2(x,l_1h)\;,$\\
$m_2(f_1h, f_1x) = f_1l_2(h, x) - f_2(l_1h, x)\;,$
\item \label{morphism_d}$m_3(f_1x,f_1y,f_1z)-m_2(f_2(x,y),f_1z)+m_2(f_1x,f_2(y,z))-m_2(f_1y,f_2(x,z))=$\\
$f_1l_3(x,y,z)+f_2(l_2(x,y),z)-f_2(x,l_2(y,z))+f_2(y,l_2(x,z))\;.$
\end{enumerate}
\end{prop}

\begin{cor} The category ${\tt 2Lei_{\infty}}$ of 2-term Leibniz infinity algebras and infinity morphisms is a full subcategory of the category ${\tt Lei_{\infty}}$-${\tt Alg}$ of Leibniz infinity algebras and infinity morphisms.\end{cor}

\subsection{From the Kan property to 2-term infinity homotopies and their compositions}\label{KanHomComp}

\begin{defi} A \emph{2-term infinity homotopy} between infinity morphisms $f=(f_1,f_2)$ and $g=(g_1,g_2)$, which act themselves between 2-term Leibniz infinity algebras $(V,l_1,l_2,l_3)$ and $(W,m_1,m_2,m_3)$, is a linear map $\theta_1$ from $V$ to $W$, of degree $|\theta_1|=1$, which verifies, for any $x,y\in V_0$ and $h\in V_1$,
\begin{enumerate}[(a)]\label{fdg}
\item $g_1x-f_1x=m_1\theta_1x\;,$ \label{homotopy_a}
\item $g_1h-f_1h=\theta_1l_1h\;,$ \label{homotopy_b}
\item \label{homotopy_c} $g_2(x,y)-f_2(x,y)=\theta_1l_2(x,y)-m_2(f_1x,\theta_1y)-m_2(\theta_1x,g_1y)\;.$
\end{enumerate}
\label{HomTheo}\end{defi}

The characterizing relations (a) - (c) of infinity Leibniz homotopies are the correct counterpart of the defining relations of infinity Lie homotopies \cite{BC04}. However, rather than choosing the preceding relations as a mere definition, we deduce them here from the Kan property of $\op{MC}_\bullet(\bar L)$. More precisely,

\begin{theo} There exist surjective maps $S^{\,i}_1$, $i\in\{0,1\},$ from the class $\cal I$ of $\infty$-homotopies for 2-term Leibniz infinity algebras to the class $\cal T$ of $\,2$-term $\infty$-homotopies for 2-term Leibniz infinity algebras.\label{KanHom1}\end{theo}

\begin{rem} The maps $S^{\,i}_1$ preserve the source and the target, i.e. they are surjections from the class ${\cal I}(f,g)$ of $\;\infty$-homotopies from $f$ to $g$, to the class ${\cal T}(f,g)$ of $\;2$-term $\,\infty$-homotopies from $f$ to $g$. In the sequel, we refer to a preimage by $S_1^i$ of an element $\zy_1\in{\cal T}$ as a \emph{lift} of $\zy_1$ by $S_1^i$ .\end{rem}

\begin{proof} Henceforth, we use again the homological version of infinity algebras ($k$-ary bracket of degree $k-2$), as well as the Lada-Stasheff sign convention for the higher Jacobi conditions and the {\small MC} equation. \medskip

Due to the choice of the homological variant of homotopy algebras, $\zd={\cal L}_1$ has degree $-1$. For consistency, differential forms are then viewed as negatively graded; hence, $d:\zW^{-k}(\zD^n)\to \zW^{-k-1}(\zD^n)$, $k\in\{0,\ldots,n\}$, and $\bar{\cal L}_1=\zd\0\id+\id\0\, d$ has degree $-1$ as well. Similarly, the degree of the operator $h_n^i$ is now $|h_n^i|=1$. It is moreover easily checked that $L$ cannot contain multilinear maps of nonnegative degree, i.e. that $L=\oplus_{k\ge 0}L_{-k}$. It follows that an element $\bar\za\in(L\0\zW^\star(\zD^n))^{-k}$, $k\ge 0$, reads $$\bar\za=\sum\za_{-k}\0\zw^0+\sum\za_{-k+1}\0\zw^{-1}+\ldots\;,$$ where the {\small RHS} is a finite sum. For instance, if $n=2$, an element $\bar\za$ of degree $-1$ can be decomposed as $$\bar\za=\za(s,t)\0 1+\zb(s,t)\0 ds+\zb'(s,t)\0 dt\;,$$ where $(s,t)$ are coordinates of $\zD^2$ and where $\za(s,t)\in L_{-1}[s,t]$ and $\zb(s,t),\zb'(s,t)\in L_0[s,t]$ are polynomial functions in $s,t$ with coefficients in $L_{-1}$ and $L_0$, respectively.\medskip

In the sequel, we evaluate the $L_{\infty}$-structure maps $\bar{\cal L}_i$ of $L\0 \zW^\star(\zD^n)$ mainly on elements of degree $-1$ and $0$, hence we compute the structure maps ${\cal L}_i$ of $L=\h_\K(\op{Zin}^c(sV),W)$ on elements $\za$ and $\zb$ of degree $-1$ and $0$, respectively. Let
\bea \ga=\sum_{p\geq1}\ga^p\in L\;,~~~|\ga|=-1\;,\nn\\
\zb=\sum_{p\geq1}\zb^p\in L\;,~~~|\zb|=0\;,\nn\eea
where $\ga^p,\zb^p:~(sV)^{\otimes p}\to W$. The point is that the concentration of $V,W$ in degrees $0,1$ entails that almost all components $\za^p,\zb^p$ vanish and that all series converge (which explains why the formal application of Getzler's method to the present situation leads to the correct counterpart of the findings of \cite{BC04}). Indeed, the only nonzero components of $\za,\zb$ are
\bea \ga^1:&& sV_0\to W_0,~~sV_1\to W_1\;,\nn\\
\ga^2:&&(sV_0)^{\otimes 2}\to W_1\;,\nn\\
\zb^1:&&sV_0\to W_1\;.\label{CompAlphaBeta}\eea
Similarly, the nonzero components of the nonzero evaluations of the maps ${\cal L}_i$ on $\ga$-s and $\zb$-s are
\bea
\SC{L}_1(\ga):&&sV_1\to W_0\;,~~(sV_0)^{\otimes 2} \to W_0\;,~~sV_0\otimes sV_1\to W_1\;,~~(sV_0)^{\otimes 3}\to W_1\;,\nn\\
\SC{L}_1(\zb):&& sV_0\to W_0\;,~~sV_1\to W_1\;,~~(sV_0)^{\otimes 2}\to W_1\;,\nn\\
\SC{L}_2(\ga_1,\ga_2):&& (sV_0)^{\otimes 2}\to W_0\;,~~sV_0\otimes sV_1\to W_1\;,~~(sV_0)^{\otimes 3}\to W_1\;,\nn\\
\SC{L}_2(\ga,\zb):&&(sV_0)^{\otimes 2}\to W_1\;,\nn\\
\SC{L}_3(\ga_1,\ga_2,\ga_3):&& (sV_0)^{\otimes 3}\to W_1\;,\label{non_zero_maps}\eea see Proposition \ref{proposition_about_lie_inf_on_space_of_homomorphisms}.\medskip

We are now prepared to concretize the iterative construction of ${\cal B}_n^i(\zm,\zn)\in\op{MC}_n(\bar L)$ from $\zm\in\op{MC}(L)$ and $\zn=(\zd+d)\zb$, $\zb\in(L\0\zW^\star(\zD^n))^0$, $\ze_n^i\zb=0$ (the explicit forms of ${\cal B}_n^i(\zm,\zn)$ for $n=1$ and $n=2$ will be the main ingredients of the proofs of Theorems \ref{KanHom1} and \ref{KomComp2}).\medskip

$\bullet\quad$ Let $\za\in{\cal I}(f,g)$, i.e. let $$\za\in\op{MC}_1(\bar L)\stackrel{\sim}{\longrightarrow}(\zm,(\zd+d)\zb)\in\op{MC}(L)\times\op{mc}_1^0(\bar L)\;,$$ such that $\ze_1^0\za=f$ and $\ze_1^1\za=g$. To construct $$\za={\cal B}_1^0B_1^0\za={\cal B}_1^0(\ze_1^0\za,(\zd+d)h_1^0\za)=:{\cal B}_1^0(f,(\zd+d)\zb)=:{\cal B}_1^0(\zm,\zn)\;,$$ we start from
\bea \ga_0=\mu+(\gd+d)\zb\;.\nn\eea
The iteration unfolds as
\bea \ga_{k}=\ga_0-\sum_{j=2}^{\infty}\frac{1}{j!}h_1^0\bar{\SC{L}}_j(\ga_{k-1},\cdots,\ga_{k-1})\;,\quad k\ge 1\;.\nn\eea
Explicitly,
\bea \ga_1&=&\mu+(\gd+d)\zb-\frac12h_1^0\bar{\SC{L}}_2(\ga_0,\ga_0)-\frac1{3!}h_1^0\bar{\SC{L}}_3(\ga_0,\ga_0,\ga_0)\nn\\
&=&\mu+(\gd+d)\zb-h_1^0\bar{\SC{L}}_2(\zm+\gd \zb,d\zb)-\frac1{2}h_1^0\bar{\SC{L}}_3(\mu+\gd \zb,\mu+\gd \zb,d\zb)\nn\\
&=&\mu+(\gd+d)\zb-h_1^0\bar{\SC{L}}_2(\mu+\gd \zb,d\zb)\;.\nn\eea
Observe that $\zm+\zd\zb\in L_{-1}[t]$ and $d\zb\in L_0[t]\0 dt$, that differential forms are concentrated in degrees $0$ and $-1$, that $h_1^0$ annihilates 0-forms, and that the term in $\bar{\cal L}_3$ contains a factor of the type ${\cal L}_3(\za_1,\za_2,\zb)$ (notation of (\ref{non_zero_maps})), whose components vanish -- see above. 
Analogously,
\bea \ga_2&=&\mu+(\gd+d)\zb-h_1^0\bar{\SC{L}}_2(\mu+\gd \zb-h_1^0\bar{\SC{L}}_2(\mu+\gd \zb,d\zb),d\zb)\nn\\
&&-\frac12h_1^0\bar{\SC{L}}_3(\mu+\gd \zb-h_1^0\bar{\SC{L}}_2(\mu+\gd \zb,d\zb),\mu+\gd \zb-h_1^0\bar{\SC{L}}_2(\mu+\gd \zb,d\zb),d\zb)\nn\\
&=&\mu+(\gd+d)\zb-h_1^0\bar{\SC{L}}_2(\mu+\gd \zb,d\zb)\;.\nn\eea
Indeed, the term $h_1^0\bar{\cal L}_2(h_1^0\bar{\cal L}_2(\zm+\zd\zb,d\zb),d\zb)$ contains a factor of the type ${\cal L}_2({\cal L}_2(\za,\zb_1),\zb_2)$ (notation of (\ref{non_zero_maps})), and the only nonvanishing component of this factor, as well as of its first internal map ${\cal L}_2(\za,\zb_1)$, is the component $(sV_0)^{\0 2}\to W_1$ -- which entails, in view of Proposition \ref{proposition_about_lie_inf_on_space_of_homomorphisms}, that the considered term vanishes. Hence, the iteration stabilizes already at its second stage and \be\label{HomotDef}\za={\cal B}_1^0(\zm,\zn)=\mu+(\gd+d)\zb-h_1^0\bar{\SC{L}}_2(\mu+\gd \zb,d\zb)\in\op{MC}_1(\bar L)\;.\ee



Remark first that the integral $h_1^0$ can be evaluated since $\bar{\SC{L}}_2(\mu+\gd \zb,d\zb)$ is a total derivative. Indeed, when setting $\zb=\zb_0\0 P$ (sum understood), $\zb_0\in L_0$ and $P\in\zW^0(\zD^1)$, we see that $$\bar{\cal L}_2(\zm,d\zb)={\cal L}_2(\zm,\zb_0)\0 dP=-d\bar{\cal L}_2(\zm,\zb)\;.$$ As for the term $\bar{\cal L}_2(\zd\zb,d\zb)$, we have \bea 0 = (\zd+d)\bar{\cal L}_2(\zb,d\zb)=\bar{\SC{L}}_2(\gd \zb, d\zb)+\bar{\SC{L}}_2(\zb,\gd d\zb)\;,\nn\eea since $\bar{\cal L}_1=\zd+d$ is a graded derivation of $\bar{\cal L}_2$ and as $\bar{\cal L}_2(\zb,d\zb)=\bar{\cal L}_2(d\zb,d\zb)=0$. It is now easily checked that \bea \bar{\SC{L}}_2(\gd\zb, d\zb)=-\frac12d\bar{\SC{L}}_2(\gd \zb, \zb)\;.\nn\eea
Eventually,
\bea \ga&=&\mu+(\gd+d)\zb+h_1^0d\bar{\SC{L}}_2(\mu,\zb)+\frac12h_1^0d\bar{\SC{L}}_2(\gd \zb,\zb)\nn\\
&=&\mu+(\gd+d)\zb+\bar{\SC{L}}_2(\mu,\zb)+\frac12\bar{\SC{L}}_2(\gd \zb,\zb)\;.\nn\eea Indeed, it suffices to observe that, for any $\ell_{-1}\0 P\in L_{-1}\0\zW^0(\zD^1)$ which vanishes under the action of $\ze_1^0$, we have \bea h_1^0d(\ell_{-1}\0 P)=-dh_1^0(\ell_{-1}\0 P)+\ell_{-1}\0 P-\ze_1^0(\ell_{-1}\0 P)=\ell_{-1}\0 P\;.\nn\eea
We are now able to write the components of $g=\ze_1^1\za\in L_{-1}$ (see (\ref{CompAlphaBeta})) in terms of $f=\zm$ and $\zb$:
\bea
&&g^1=\ep_1^1(\mu+(\gd+d)\zb-\bar{\SC{L}}_2(\mu,\zb)-\frac12\bar{\SC{L}}_2(\gd \zb,\zb))^1=\ep_1^1(f+\gd \zb)^1=f^1+\ep_1^1(\gd \zb)^1\;,\nn\\
&&g^2=\ep_1^1(\mu+(\gd+d)\zb-\bar{\SC{L}}_2(\mu,\zb)-\frac12\bar{\SC{L}}_2(\gd \zb,\zb))^2=
\ep_1^1(f+\gd \zb-\bar{\SC{L}}_2(f,\zb)-\frac12\bar{\SC{L}}_2(\gd \zb,\zb))^2\;,\nn\\
&&g^3=0\;,\label{CharRelHomotPreFin}\eea
where we changed signs according to our sign conventions and remembered that the first component of a morphism of the type ${\cal L}_2(\za,\zb)$ (see (\ref{non_zero_maps})) vanishes.\medskip

To obtain a 2-term $\infty$-homotopy $\zy_1\in{\cal T}(f,g)$, it now suffices to further develop the equations (\ref{CharRelHomotPreFin}).

As $$g_1:=g^1s,f_1:=f^1s\in \h_\K^{0}(V,W)\;,$$ we evaluate the first equation on $x\in V_0$ and $h\in V_1$. Therefore, we compute $\ze_1^1(\zd\zb)^1s=\zd\zb(1)^1s$ on $x$ and $h$. Since $$\zd\zb(1)={\cal L}_1\zb(1)=m_1\zb(1)+\zb(1) D_V\;,$$ where $D_V\in \op{CoDer}^{-1}(\op{Zin}^c(sV))$, we have $D_V:sV_1\to sV_0\;, sV_0\0 sV_0\to sV_0\;,\ldots$ Hence, \be\label{HomotDefIntermed}\zd\zb(1)^1sx = m_1\,\zb(1)s\,x = m_1\zy_1x\;,\ee where we defined the {\it homotopy parameter} $\zy_1$ by \be\label{HomotPara}\zy_1:=\zb(1)s=\zb(1)s-\zb(0)s\;.\ee Similarly, \be\label{HomotDefIntermed1}\zd\zb(1)^1sh=\zb(1)D_Vsh=\zb(1)s\,s^{-1}D_Vs\,h=\zy_1l_1h\;.\ee The characterizing equations (a) and (b) follow.

Since $$g_2:=g^2s^2,f_2:=f^2s^2\in\h_\K^1(V\0 V,W)\;,$$ it suffices to evaluate the second equation on $x,y\in V_0$. When computing e.g. $\ze_1^1\bar{\cal L}_2(\zd\zb,\zb)^2s^2(x,y)$, we get $${\cal L}_2(\zd\zb(1),\zb(1))(sx,sy)=m_2(\zd\zb(1)\,sx,\zb(1)\,sy)+m_2(\zb(1)\,sx,\zd\zb(1)\,sy)=$$ \be\label{HomotDefIntermed2}m_2(m_1\zy_1x,\zy_1y)+m_2(\zy_1x,m_1\zy_1y)=2m_2(\zy_1x,m_1\zy_1y)\;,\ee in view of Equation (\ref{HomotDefIntermed}) and Relation (b) of Proposition \ref{2term_Loday_in_algebra}. Similarly, \be\label{HomotDefIntermed3}\ze_1^1\bar{\SC{L}}_2(f,\zb)^2s^2(x,y)=m_2(f_1x,\zy_1y)+m_2(\zy_1x,f_1y)\;.\ee Further, one easily finds \be\label{HomotDefIntermed4} \ze_1^1(\zd\zb)^2s^2(x,y)=\zy_1l_2(x,y)\;.\ee When collecting the results (\ref{HomotDefIntermed2}), (\ref{HomotDefIntermed3}), and (\ref{HomotDefIntermed4}), and taking into account Relation (a), we finally obtain the characterizing equation (c).\bigskip

$\;\bullet\quad$ Recall that in the preceding step we started from $\za\in {\cal I}(f,g)$, set $\zm=f$, $$\zb=h_1^0\za\;,$$ $\zn=(\zd+d)\zb$, defined $$\zy_1=(\zb(1)-\zb(0))s\;,$$ and deduced the characterizing relations $g=f+{\cal E}(f,\zb(1)s)=f+{\cal E}(f,\zy_1)$ of $\zy_1\in{\cal T}(f,g)$ by computing $$\za={\cal B}^0_1(\zm,\zn)=\mu+(\gd+d)\zb+\bar{\SC{L}}_2(\mu,\zb)+\frac12\bar{\SC{L}}_2(\gd \zb,\zb)$$ at $1$. Let us mention that instead of defining the map $S_1^0:{\cal I}\ni\za\mapsto \zy_1\in {\cal T}$, we can consider the similarly defined map $S_1^1$.\medskip

To prove surjectivity of $S_1^i$, let $\zy_1\in{\cal T}(f,g)$ and set $\zb(i)=0$, $i\in\{0,1\}$, and $\zb(1-i)=(-1)^{i}\zy_1s^{-1}$. Note that by construction $\zy_1=(\zb(1)-\zb(0))s$. Use now Renshaw's method \cite{Sul77} to extend $\zb(0)$ and $\zb(1)$ to some $\zb\in L_0\0\zW^0(\zD^1)$, set $$\zm=(1-i)f+ig\quad\text{and}\quad\zn=(\zd+d)\zb\;,$$ and construct \be\label{preimage}\za={\cal B}_1^i(\zm,\zn)\in\op{MC}_1(\bar L)\;.\ee If $i=0$, then $$\za(0)=\ze_1^0\za=\zm=f\quad\text{and}\quad\za(1)=({\cal B}_1^0(\zm,\zn))(1)=f+{\cal E}(f,\zy_1)=g\;,$$ in view of the characterizing relations (a)-(c) of $\zy_1$. If $i=1$, one has also $\za(1)=g$ and $\za(0)=g+{\cal E}(g,-\zy_1)=f$, but to obtain the latter result, the characterizing equations (a)-(c), as well as Equation (b) of Proposition \ref{2term_Loday_in_algebra} are needed. To determine the image of $\za\in{\cal I}(f,g)$ by $S_1^i$, one first computes $h_1^i\za$, which, since $h_1^i$ sends 0-forms to 0, is equal to $$h_1^i(\zd+d)\zb=-(\zd+d)h_1^i\zb+\zb-\ze_1^i\zb=\zb\;,$$ then one gets $$S_1^i\za=(\zb(1)-\zb(0))s=\zy_1\;,$$ which completes the proof.\end{proof}

\begin{theo}[Definition] If $\;\theta_1:f\Rightarrow g$, $\zt_1:g\Rightarrow h$ are 2-term $\infty$-homotopies between infinity morphisms $f,g,h:V\to W$, the vertical composite $\zt_1\circ_1 \theta_1$ is given by $\zt_1+\theta_1$.\label{KomComp2}\end{theo}

We will actually lift $\zy_1,\zt_1\in{\cal T}$ to $\za',\za''\in\op{MC}_1(\bar L)$ (which involves choices), then compose these lifts in the infinity groupoid $\op{MC}_\bullet(\bar L)$ (which is not a well-defined operation), and finally project the result back to ${\cal T}$ (despite all the intermediate choices, the final result will turn out to be well-defined).

\begin{proof} Let now $n=2$, take $\mu\in\textrm{MC}(L)$ and $\zn=(\zd+d)\zb\in\op{mc}_2^1(\bar L)$, then construct $\za={\cal B}_2^1(\zm,\zn)$. The computation is similar to that in the 1-dimensional case and gives the same result:
\bea \ga=\mu+(\gd+d)\zb+\bar{\SC{L}}_2(\mu,\zb)+\frac12\bar{\SC{L}}_2(\gd \zb,\zb)\;.\label{HomotCompIntermed}\eea

To obtain $\zt_1\circ_1\zy_1$, proceed as in (\ref{preimage}) and lift $\zy_1$ (resp., $\zt_1$) to $$\za':={\cal B}_1^1(g,(\zd+d)\zb')\in{\cal I}(f,g)\subset\op{MC}_1(\bar L)\quad (\text{resp.,}\;\; \za'':={\cal B}_1^0(g,(\zd+d)\zb'')\in{\cal I}(g,h)\subset\op{MC}_1(\bar L))\;,$$ where $$\zb'(0)=-\zy_1s^{-1}\;\text{and}\;\zb'(1)=0\quad (\text{resp.,}\;\; \zb''(0)=0\;\text{and}\;\zb''(1)=\zt_1s^{-1})\;.$$ As mentioned above, we have by construction \be\label{Thetas} \zy_1=(\zb'(1)-\zb'(0))s\quad(\text{resp.,}\;\;\zt_1=(\zb''(1)-\zb''(0))s)\;.\ee

If we view $\za'$ (resp., $\za''$) as defined on the face $01$ (resp., $12$) of $\zD^2$, the equation $\ze_1^1\za'=\ze_1^0\za''=g$ reads $\ze_2^1\za'=\ze_2^1\za''=g=:\zm$. This means that $$(\za',\za'')\in\op{SSet}(\zL^1[2],\op{MC}_\bullet(\bar L))\;.$$ We now follow the extension square (\ref{commute_square}). The left arrow leads to $$(\zm;(\zd+d)\zb',(\zd+d)\zb'')\in\op{SSet}(\zL^1[2],\op{MC}(L)\times\op{mc}_\bullet(\bar L))\;,$$ the bottom arrow to $$(\zm,(\zd+d)\zb)\in\op{MC}(L)\times\op{mc}_2^1(\bar L)\;,$$ where $\zb$ is {\it any extension} of $(\zb',\zb'')$ to $\zD^2$, and the right arrow provides $\za\in\op{MC}_2(\bar L)$ given by Equation (\ref{HomotCompIntermed}). From Subsection \ref{InftyCatComp}, we know that all composites of $\za',\za''$ are $\infty$-2-homotopic and that a possible composite is obtained by restricting $\za$ to $02.$ This restriction $(-)|_{02}$ is given by the ${\tt DGCA}$-map $d_1^2$. Hence, we get $$\ga|_{02}=\mu+(\gd+d)\zb|_{02}+\bar{\SC{L}}_2(\mu,\zb|_{02})+\frac12\bar{\SC{L}}_2(\gd \zb|_{02},\zb|_{02})\in{\cal I}(f,h)\subset\op{MC}_1(\bar L)\;.$$

We now choose the projection $S_1^0\za|_{02}\in{\cal T}(f,h)$ of the composite-candidate of the chosen lifts of $\zy_1,\zt_1$, as composite $\zt_1\circ_1\zy_1$. Since $$h_1^0\za|_{02}=-(\zd+d)h_1^0\zb|_{02}+\zb|_{02}-\zb(0)=\zb|_{02}-\zb(0)\;,$$ we get
$$ S_1^0\za|_{02}=(\zb|_{02}(2)-\zb(0)-\zb|_{02}(0)+\zb(0))s=(\zb(2)-\zb(0))s=$$ $$(\zb''(2)-\zb''(1))s+(\zb'(1)-\zb'(0))s=\zt_1+\zy_1\;,$$ in view of (\ref{Thetas}). Hence, by definition, the vertical composite of $\zy_1\in{\cal T}(f,g)$ and $\zt_1\in{\cal T}(g,h)$ is given by \be\label{VertComp}\zt_1\circ_1\zy_1=\zt_1+\zy_1\in{\cal T}(f,h)\;.\ee
\end{proof}

\begin{rem} The composition of elements of ${\cal I}=\op{MC}_1(\bar L)$ in the infinity groupoid $\op{MC}_\bullet(\bar L)$, which is defined and associative only up to higher morphisms, projects to a well-defined and associative vertical composition in ${\cal T}$.\end{rem}

Just as for concordances, horizontal composition of $\infty$-homotopies is without problems. The horizontal composite of $\zy_1\in{\cal T}(f,g)$ and $\zt_1\in{\cal T}(f',g')$, where $f,g:V\to W$ and $f',g':W\to X$ act between 2-term Leibniz infinity algebras, is defined by \be\label{HorzComp} \zt_1\circ_0\zy_1=g'_1\zy_1+\zt_1f_1 = f'_1\zy_1+\zt_1g_1\;.\ee The two definitions coincide, since $\zy_1,\zt_1$ are chain homotopies between the chain maps $f,g$ and $f',g'$, respectively, see Definition \ref{HomTheo}, Relations (\ref{homotopy_a}) and (\ref{homotopy_b}). The identity associated to a 2-term $\infty$-morphism is just the zero-map. As announced in \cite{BC04} (in the Lie case and without information about composition), we have the

\begin{prop} There is a strict 2-category ${\tt 2Lei_{\infty}}$-${\tt Alg}$ of 2-term Leibniz infinity algebras.\end{prop}

\section{2-Category of categorified Leibniz algebras}

\subsection{Category of Leibniz 2-algebras}

Leibniz 2-algebras are categorified Leibniz structures on a categorified vector space. More precisely,

\begin{defi}
A \emph{Leibniz 2-algebra} $(L,[-,-],{\mathbf J})$ is a linear category $L$ equipped with
\begin{enumerate}
\item a \emph{bracket} $[-,-]$, i.e. a bilinear functor $[-,-]:L\times L\rightarrow L$, and
\item a \emph{Jacobiator} $\mathbf{J}$, i.e. a trilinear natural transformation
$$
\mathbf{J}_{x,y,z}:[x,[y,z]]\rightarrow[[x,y],z]+[y,[x,z]],\quad x,y,z\in L_0,
$$
\end{enumerate}
which verify, for any $w,x,y,z\in L_0,$ the \emph{Jacobiator identity}
\begin{equation}\label{Jacobiator_diagramm}
\hspace{-5mm}
\xymatrix@C=4pc@R=4pc{
&[w,[x,[y,z]]]\ar@{->}
[dr]_-{\mathbf{1}}\ar@{->}[dl]^-{[\mathbf{1}_w,\mathbf{J}_{x,y,z}]}\\
[w,[[x,y],z]]+[w,[y,[x,z]]]\ar@{->}
[d]^-{\mathbf{J}_{w,[x,y],z}+\mathbf{J}_{w,y,[x,z]}}&&
[w,[x,[y,z]]]\ar@{->}
[d]_-{\mathbf{J}_{w,x,[y,z]}}\\
{\begin{gathered}[t]
[[w,[x,y]],z]+[[x,y],[w,z]]\\
+[[w,y],[x,z]]+[y,[w,[x,z]]]
\end{gathered}}\ar@{->}
[d]^-{\mathbf{1}+[\mathbf{1}_y,\mathbf{J}_{w,x,z}]}
&&[[w,x],[y,z]]+[x,[w,[y,z]]]\ar@{->}
[d]_-{1+[\mathbf{1}_x,\mathbf{J}_{w,y,z}]}\\
{\begin{gathered}[t]
[[w,[x,y]],z]+[[x,y],[w,z]]\\
+[[w,y],[x,z]]+[y,[[w,x],z]]\\
+[y,[x,[w,z]]]
\end{gathered}}\ar@{->}
[dr]^-{[\mathbf{J}_{w,x,y},\mathbf{1}_z]}&&
{\begin{gathered}[t]
[[w,x],[y,z]]+[x,[[w,y],z]]\\
+[x,[y,[w,z]]]
\end{gathered}}\ar@{->}
[ld]^-{\mathbf{J}_{[w,x],y,z}+\mathbf{J}_{x,[w,y],z}+\mathbf{J}_{x,y,[w,z]}}\\
&{\begin{gathered}[t]
[[[w,x],y],z]+[[x,[w,y]],z]\\
+[[x,y],[w,z]]+[[w,y],[x,z]]\\
+[y,[[w,x],z]]+[y,[x,[w,z]]]
\end{gathered}}
}
\end{equation}
\end{defi}

The Jacobiator identity is a coherence law that should be thought of as a higher Jacobi identity for the Jacobiator.\medskip

The preceding hierarchy `category, functor, natural transformation' together with the coherence law is entirely similar to the known hierarchy `linear, bilinear, trilinear maps $l_1,l_2,l_3$' with the $L_{\infty}$-conditions (a)-(e). More precisely,

\begin{prop}
There is a 1-to-1 correspondence between Leibniz 2-algebras and 2-term Leibniz infinity algebras.
\end{prop}

This proposition was proved in the Lie case in \cite{BC04} and announced for the Leibniz case in \cite{SL10}. A generalization of the latter correspondence to Lie 3-algebras and 3-term Lie infinity algebras can be found in \cite{KMP11}. This paper allows to understand that the correspondence between higher categorified algebras and truncated infinity algebras is subject to cohomological conditions, and to see how the coherence law corresponds to the last nontrivial $L_{\infty}$-condition.\medskip

The definition of Leibniz 2-algebra morphisms is God-given: such a morphism must be a functor that respects the bracket up to a natural transformation, which in turn respects the Jacobiator. More precisely,

\begin{defi}
Let $(L,[-,-],\mathbf{J})$ and $(L',[-,-]',\mathbf{J}')$ be Leibniz 2-algebras $($in the following, we write $[-,-], \mathbf{J}$ instead of $\;[-,-]',\mathbf{J}'$$)$. A \emph{morphism $(F,\mathbf{F})$ of Leibniz 2-algebras} from $L$ to $L'$ consists of
\begin{enumerate}
\item a linear functor $F:L\rightarrow L'$, and
\item a bilinear natural transformation $$\mathbf{F}_{x,y}:[Fx,Fy]\rightarrow F[x,y],\quad x,y\in L_0\;,$$\end{enumerate}
which make the following diagram commute
\begin{equation}\label{diagram_for_morphism}
\xymatrix@C=10pc@R=3pc{
[Fx,[Fy,Fz]]\ar@{->}[d]^{[\mathbf{1}_x,\mathbf{F}_{y,z}]}\ar@{->}[r]^-{\mathbf{J}_{Fx,Fy,Fz}}&[[Fx,Fy],Fz]+[Fy,[Fx,Fz]]\ar@{->}[d]^{[\mathbf{F}_{x,y},\mathbf{1}_z]+[\mathbf{1}_y,\mathbf{F}_{x,z}]}\\
[Fx,F[y,z]]\ar@{->}[d]^{\mathbf{F}_{x,[y,z]}}&[F[x,y],Fz]+[Fy,F[x,z]]\ar@{->}[d]^{\mathbf{F}_{[x,y],z}+\mathbf{F}_{y,[x,z]}}\\
F[x,[y,z]]\ar@{->}[r]^{F\mathbf{J}_{x,y,z}}&F[[x,y],z]+F[y,[x,z]]
}
\end{equation}
\end{defi}
\begin{prop}
\vspace{2mm}
There is a 1-to-1 correspondence between Leibniz 2-algebra morphisms and 2-term Leibniz infinity algebra morphisms.
\end{prop}

For a proof, see \cite{BC04} and \cite {SL10}.\medskip

Composition of Leibniz 2-algebra morphisms $(F,\mathbf{F})$ is naturally given by composition of functors and whiskering of functors and natural transformations.

\begin{prop} There is a category ${\tt Lei2}$ of Leibniz 2-algebras and morphisms.\end{prop}

\subsection{2-morphisms and their compositions}

The definition of a 2-morphism is canonical:

\begin{defi}
Let $(F,\mathbf{F}),(G,\mathbf{G})$ be Leibniz 2-algebra morphisms from $L$ to $L'$. A \emph{Leibniz 2-algebra 2-morphism} $\boldsymbol\theta$ from $F$ to $G$ is a linear natural transformation $\boldsymbol{\theta}: F\Rightarrow G$, such that, for any $x,y\in L_0,$ the following diagram commutes
\begin{equation}\label{diagram_for_2_morphism}
\xymatrix@C=2pc@R=2pc{
[Fx,Fy]\ar@{->}[rr]^-{\mathbf{F}_{x,y}}\ar@{->}[dd]^-{[\boldsymbol{\theta}_x,\boldsymbol{\theta}_y]}&&F[x,y]\ar@{->}[dd]^-{\boldsymbol{\theta}_{[x,y]}}\\
\\
[Gx,Gy]\ar@{->}[rr]^-{\mathbf{G}_{x,y}}&&G[x,y]\\
}
\end{equation}
\end{defi}

\begin{theo}
There is a 1:1 correspondence between Leibniz 2-algebra 2-morphisms and 2-term Leibniz $\infty$-homotopies.
\end{theo}

Horizontal and vertical compositions of Leibniz 2-algebra 2-morphisms are those of natural transformations.

\begin{prop} There is a strict 2-category ${\tt Lei2Alg}$ of Leibniz 2-algebras.\end{prop}

\begin{cor} The 2-categories ${\tt 2Lei_{\infty}}$-${\tt Alg}$ and ${\tt Lei2Alg}$ are 2-equivalent.\end{cor}

\end{document}